\documentclass[a4paper,12pt,twoside]{amsart}
\usepackage[english]{babel}
\usepackage{amssymb, amsmath, eucal, enumerate}
\usepackage[all]{xy}

\usepackage[left=3cm,right=3cm,top=3cm,bottom=3cm]{geometry}

\usepackage[svgnames,hyperref]{xcolor}
\newcommand{\mycolor}{Navy}
\usepackage[pdfauthor={Son H. Do and  Dat T. T\^O}, colorlinks, linktocpage, citecolor = \mycolor, linkcolor = \mycolor, urlcolor = \mycolor]{hyperref}
\newtheorem{The}{Theorem}[section]
\newtheorem{Lem}[The]{Lemma}
\newtheorem{Prop}[The]{Proposition}
\newtheorem{Cor}[The]{Corollary}
\newtheorem{Rem}[The]{Remark}

\newtheorem{Def}[The]{Definition}

\newcommand{\C}{\mathbb{C}}
\newcommand{\R}{\mathbb{R}}
\newcommand{\N}{\mathbb{N}}
\newcommand{\Z}{\mathbb{Z}}

\newcommand{\E}{\mathcal{E}}
\newcommand{\F}{\mathcal{F}}

\newcommand{\dt}{\partial_t}
\newcommand{\MAu}{(dd^cu)^n}

\begin{document}
 \title[Parabolic Monge-Amp\`ere equations]{Viscosity solutions
 	to parabolic complex Monge-Amp\`ere equations} 
\setcounter{tocdepth}{1}
\author{Hoang-Son Do} 
\address{Institute of Mathematics \\ Vietnam Academy of Science and Technology \\18
Hoang Quoc Viet \\Hanoi \\Vietnam}
\email{hoangson.do.vn@gmail.com, dhson@math.ac.vn}
\author{Giang Le}
\address{Department of Mathematics, Hanoi National University of Education, 136-Xuan Thuy, Cau Giay, Hanoi, VietNam}
\email{legiang@hnue.edu.vn, legiang01@yahoo.com}
\author{Tat Dat T\^o}
\address{Ecole Nationale de l'Aviation Civile and Institut de Mathématiques de Toulouse (Chercheur associ\'e), Unversit\'e F\'ed\'erale de Toulouse\\
7, Avenue Edouard Belin\\
FR-31055 Toulouse Cedex}
\email{tat-dat.to@enac.fr, tatdat.to@gmail.com}
\date{\today\\The first-named author was funded by the Vietnam National Foundation for Science and Technology Development (NAFOSTED) under grant number 101.02-2017.306.}
\maketitle
\begin{abstract}
In this paper, we study the Cauchy-Dirichlet  problem for  Parabolic complex Monge-Amp\`ere equations  on a strongly pseudoconvex domain  using the viscosity method. We extend the  results in \cite{EGZ15b} on the existence of solution and the convergence 
at infinity. We also establish the H\"older   regularity  of  the solutions when the Cauchy-Dirichlet data are H\"older continuous. 
\end{abstract}
\section{Introduction}
In \cite{ST,ST12}, Song and Tian gave a conjectural picture to approach the Minimal Model Program via the K\"ahler-Ricci flow. In the Song-Tian program, one  need to  study  the  behavior of the K\"ahler-Ricci flow on mildly  singular varieties.  This requires  a theory of weak solutions for certain  degenerate parabolic complex Monge-Amp\`ere equations modelled on   
\begin{equation} \label{PMA_0}
(dd^c u)^n=e^{\dt u(t,z) + F(t,z,u)} \mu,
\end{equation}
where $\mu $ is a volume form, and  $u$ is $t$-dependent K\"ahler potential on a compact K\"ahler manifold.  

\medskip
A viscosity  approach for complex elliptic    Monge-Amp\`ere equations was  established  in   \cite{EGZ,EGZ15,Wa12, HL} (see also  \cite{DDT} for a  recent generalization). In \cite{EGZ,EGZ15}, Eyssidieux-Guedj-Zeriahi modified  the arguments  in  \cite{IL90,CIL92} to prove  local and global comparison principles  and solve  complex Monge-Amp\`ere equations  on both compact K\"ahler manifolds and complex domains.   They also compared the  viscosity solutions  and  the plurisubharmonic solutions.  On the other hand, using  the  contact set techniques (the Alexandrov–Bakelman–Pucci estimate),  Wang \cite{Wa12} also proved a local viscosity comparison principle and applied to the  Dirichlet problem for  Monge-Amp\`ere equations without using the argument  in \cite{IL90,CIL92}. In \cite{HL}, Harvey and Lawson also used the  viscosity method to the Dirichlet problem for the {\sl homogeneous} complex Monge-Amp\`ere equation on smooth hyperconvex domains in Stein manifolds.

\medskip
A complementary viscosity approach for  degenerate  {\sl parabolic} Monge-Amp\`ere equations has been developed recently  by  Eyssidieux-Guedj-Zeriahi
\cite{EGZ15b} in domains of $\C^n$ and \cite{EGZ16,EGZ18} on compact K\"ahler manifolds (see also \cite{To19} for a generalization).  This approach  also adapted the viscosity approach developped by Lions et al. (see \cite{IL90,CIL92}) to the complex case, using the elliptic side of the theory which was developed in \cite{EGZ}.

\medskip
In \cite{EGZ15b}, Eyssidieux-Guedj-Zeriahi studied  a  Cauchy-Dirichlet  problem  for \eqref{PMA_0} in which  the density $\mu$ and  parabolic boundary condition  are  independent of time.  They  proved  that in this case  the  Cauchy-Dirichlet  problem  for \eqref{PMA_0} admits a solution if the problem is {\sl admissible} (see below).  

In this note,  we solve  a more general Cauchy-Dirichlet  problem on a pseudoconvex domain for \eqref{PMA_0} in which the density $\mu$ and the  parabolic boundary condition  depend on time.  We also analyzed  the solvability of the  Cauchy-Dirichlet  problem by giving several characterizations for the admissible condition.  
In addition, we  establish the H\"older regularity of the solutions. Finally we prove that the solution  of the Cauchy-Dirichlet  problem converges, as $t\rightarrow \infty$, to  the solution of the corresponding Dirichlet  problem, extending the convergence result in \cite{EGZ15b}.

\medskip
There is a well established pluripotential theory of weak solutions to elliptic complex Monge-Amp\`ere equations, following the pionneering work of Bedford and Taylor \cite{BT1,BT2} in local case, but the similar theory for the parabolic side only developed recently \cite{GLZ1,GLZ2}. It is very  interesting to compare the viscosity and pluripotential concepts, this requires Theorem \ref{mainexistence} below. We refer the reader to \cite{GLZ3} for more details.

\medskip
We now explain the precise context. Let $\Omega\subset\C^n$ be a smooth bounded strongly
 pseudoconvex domain and $T\in (0, \infty)$. We consider the Cauchy-Dirichlet problem
\begin{equation}\label{PMA}
\begin{cases}
e^{\dt u+F(t, z, u)}\mu(t, z)=\MAu\qquad\mbox{ in }\qquad \Omega_T,\\
u=\varphi\qquad\mbox{in}\qquad [0, T)\times\partial\Omega,\\
u(0, z)=u_0(z)\qquad\mbox{in}\quad\bar{\Omega},
\end{cases}
\end{equation}
where
\begin{itemize}
	\item  $\Omega_T=(0, T)\times\Omega$.
	\item $F(t, z, r)$ is continuous in $[0, T]\times\bar{\Omega}\times\R$ and non-decreasing in $r$.
	\item $\mu(t, z)=f(t, z)dV$,
	 where $dV$ is the standard volume form in $\C^n$ and $f\geq 0$ is a bounded  continuous function in $[0, T]\times\Omega$.
	 \item $\varphi(t, z)$ is a continuous function in $[0, T]\times\partial\Omega$.
	 \item $u_0(z)$ is continuous in $\bar{\Omega}$ and plurisubharmonic in $\Omega$ such that $u_0(z)=\varphi(0, z)$ in $\partial\Omega$.
\end{itemize}

By \cite[Definition 5.6]{EGZ15b}, $(u_0, \mu (0, .))$ is called {\sl admissible} if for all $\epsilon>0$,
 there exists $u_{\epsilon}\in PSH(\Omega)\cap C(\bar{\Omega})$ and $C_{\epsilon}>0$ such that $u_0\leq u_{\epsilon}\leq u_0+\epsilon$ and $(dd^c u_{\epsilon})^n\leq 
e^{C_{\epsilon}}\mu(0, z)$ in the viscosity sense. We observe below that the condition  $u_\epsilon \in PSH(\Omega)$ is redundant (see Theorem \ref{mainadmissible} (i)). We still use the  {\sl admissible} term for  the same definition above without this condition.  
\begin{Def}\label{def admissible}
 We say that $(u_0, \mu (0, .))$ is
{\sl  admissible} if for all $\epsilon>0$, there
exist $u_{\epsilon}\in C(\bar{\Omega})$ and $C_{\epsilon}>0$ such that $u_0\leq u_{\epsilon}\leq u_0+\epsilon$ and $(dd^c u_{\epsilon})^n\leq 
e^{C_{\epsilon}}\mu(0, z)$ in the viscosity sense.  
\end{Def}
\medskip
It follows from \cite{EGZ15b} that if $\varphi, \mu$ are independent of $t$ and 
$(u_0, \mu)$ is admissible then \eqref{PMA} admits a solution. In this paper,
we extend  this result to the case in which
$\varphi$ and $\mu$  depend on $t$ as well.  Our first main result is the following
\begin{The}\label{mainexistence} Assume that $\mu=fdV$ where $f$ is independent of $t$. 
 If $(u_0(z), \mu)$ is admissible then the equation \eqref{PMA} admits a 
	unique solution.
\end{The}
Although $F$ and $f$ play different roles, they are not entirely disjoint. If 
$f\in C([0, T]\times\overline{\Omega})$ and $f>0$ then the problem can be reduced to the case $f=1$. 
If $f\in C([0, T]\times\overline{\Omega})$ and $\dfrac{f(t, z)}{f(0, z)}$ can be extended to a continuous positive function on
$[0, T]\times\overline{\Omega}$ then the problem can be reduced to the case where $f$ is independent of $t$. 
In this paper, we also  obtain the existence result to certain cases where $f$ depends on $t$ as well. Moreover, the set
$\{z: f(t, z)=0\}$ may also depend on $t$ (and then the problem can not be reduced to the case where $f$ is independent of $t$).
 We refer to Theorem
\ref{the exist2} and Corollary \ref{cor:exits_2} for more details.

\medskip
We are now interested in the notion of admissible data. In order to avoid confusion, we always denote 
by $(dd^c\phi)_P^n$ the Monge-Amp\`ere measure in the pluripotential sense
 (see \cite{BT1, BT2, Ceg04, Blo06} for the definition)
 of a psh function $\phi$ if this measure is well-defined.
We obtain the following properties:

\begin{The}\label{mainadmissible}
Let $g\geq 0$ be a bounded continuous function in $\Omega$ and $\nu=gdV$.
Let $\phi\in PSH(\Omega)\cap C(\bar{\Omega})$. Then the following holds:
\begin{itemize}
	\item [(i)] If $(\phi, \nu)$ is  admissible then the function $u_\epsilon$ in the definition \ref{def admissible} can be taken to be psh in $\Omega$.
	\item [(ii)] Admissibility is a local property: if for every $z\in\Omega$, there exists an open neighborhood $U$ of $z$
	 such that $(\phi, \nu)$ is admissible in $U$ then $(\phi, \nu)$
	  is admissible in $\Omega$. 
	  \item [(iii)]  If $\int\limits_{\{g=0\}}(dd^c\phi)_P^n=0$
	   then $(\phi, \nu)$  is admissible.
\end{itemize}
\end{The}

In particular, when $\mu$ is independent of time, we prove that this condition is also  necessary (Corollary \ref{cor exist admi}
and Remark \ref{rem 2admissible}). However, we  also give  a counter-example in which  $\mu$ depends on $t$, the  Cauchy-Dirichlet problem admits a solution but $(u_0,\mu(0,z))$ is not admissible. 
In addition, we  prove  the following  local and integral criteria to the admissible condition.  
\begin{Cor}
If $\nu=g(z)dV\geq 0$ with $\{z\in\Omega: g(z)=0\}$ is a  negligible  set, then 
$(\phi, \nu)$ is admissible for every $\phi\in C(\overline{\Omega})\cap PSH(\Omega)$.
\end{Cor}

For degenerate (elliptic)  complex Monge-Amp\`ere equations, the H\"older regularity of pluripotential solutions  has been studied 
 intensively (we refer to \cite{GKZ,DDGHKZ} and references therein). Similar results for viscosity solutions can be found in \cite{Lu, Wa12}. Nevertheless,  to the best of our knowledge, no H\"older regularity result to  the weak solutions of parabolic complex Monge-Amp\`ere equations has been established in both  pluripotential and viscosity senses (in the non-smooth case).
 In this paper, we make a first step in this direction:
\begin{The}\label{mainLip}
	Assume that $\mu=dV$ and $u(t, z)$ is a viscosity solution to \eqref{PMA}.
	Suppose that there exist $C>0$, $0<\alpha<1$ and $0< \beta< 1/2$ such that 
	\begin{center}
		$|\varphi (t, z)-\varphi (s, w)|\leq C(|t-s|^{\alpha}+|z-w|^{2\beta}), \, \forall z, w\in\partial\Omega, t,s\in [0, T)$,
	\end{center}
\begin{center}
		$\varphi (t, z)-\varphi (s, z)\leq C(t-s), \, \forall z\in\partial\Omega,
		 0<s<t<T$,
\end{center}
 and 
\begin{center}
	$|u_0(z)-u_0(w)|\leq C|z-w|^{\beta},\, \forall z, w\in\bar{\Omega}. $
\end{center}
Suppose also that, for any $K>0$, there exists $C_K>0$ such that, 
\begin{center}
	$|{F(t, z, r)}-F(t, w, r)|\leq C_K|z-w|^{\beta}$.
\end{center}
for all $z,w\in \Omega$, $t\in [0,T), r\in [-K, K]$.
 Then, there exists $\tilde{C}>0$ such that
\begin{center}
	$|u (t, z)-u (s, w)|\leq \tilde{C}(|t-s|^{\alpha}+|z-w|^{\beta})$,
\end{center}
for all $z, w\in\Omega, t,s\in [0, T)$.

\medskip
Moreover, if $\varphi$ is Lipschitz in $t$ then $u$ is locally Lipschitz in $t$ uniformly
in $z$.
\end{The}

Finally, we prove that the   viscosity solution of the Cauchy-Dirichlet problem \eqref{PMA} asymptotically recovers the solution of the corresponding elliptic Dirichlet problem under some assumptions. This also extends the convergence result in \cite{EGZ15b}. 
\begin{The}\label{mainconvergence}
	Assume that $T=\infty$, 
	$\varphi (t, z)\rightrightarrows\varphi_{\infty} (z)$
	as $t\rightarrow\infty$ and  
	$F( t, z, r)\rightrightarrows F_{\infty}(z, r)$ in $\bar{\Omega}\times\R$
	as $t\rightarrow\infty$, where $\rightrightarrows $ denotes the uniform convergence. 
	
	Suppose that $\sup_{t\geq 0}f(t, z)\in L^{p}(\Omega)$  for some $p>1$ and $f (t, z)$ converges almost everywhere
	to a function $f_{\infty}(z)$  as $t\rightarrow\infty$. Then 
	$u(t, z)$ converges uniformly to $u_{\infty}(z)$
	as $t\rightarrow\infty$, where $u_{\infty}$ is the  unique solution of the equation
	\begin{equation}\label{MAinfty}
	\begin{cases}
	u_{\infty}\in PSH(\Omega)\cap C(\overline{\Omega}),\\
	(dd^cu_{\infty})_P^n=e^{F_{\infty}(z, u_{\infty})} f_\infty(z)dV (z)\qquad\mbox{in}\qquad\Omega,\\
	u_{\infty}=\varphi_{\infty} \qquad\mbox{in}\qquad\partial\Omega.
	\end{cases}
	\end{equation}
\end{The}
The solution  $u_\infty$ to the  elliptic Dirichlet problem above is well known to exist in the pluripotential sense in \cite{K98}.
If $f_{\infty}$ is continuous then the solution in the pluripotential sense is also
the solution  in the viscosity sense \cite{EGZ,HL,Wa12}.

In fact, we can also obtain the uniform convergence in capacity when $p=1$ as well, we refer to Theorem \ref{mainconvergence1}. In this case,
the equation \eqref{MAinfty} is replaced by the equation \eqref{MAinfty_1}. The existence of the pluripotential solution to 
\eqref{MAinfty_1} holds due to 
 \cite{Ceg04}, \cite{Aha07}, \cite{ACCP09} (see also \ref{the MA Cegrell}).

\medskip
\noindent\textbf{Acknowledgement.} The authors are grateful to  Vincent Guedj and
 Ahmed Zeriahi for very useful discussions.  This work was begun during a visit by the first-named author to the Institut de Math\'ematiques de Toulouse (from September 1 to September 30, 2018) funded by LIA Formath Vietnam and ANR GRACK. This paper was partially written while the first-named author
 visited Vietnam Institute for Advanced Study in Mathematics(VIASM). He would like to thank  
 these institutions for their hospitality.  The authors would like to thank the referee for very useful comments and suggestions.
\section{Preliminaries}
For the reader's convenience, we recall some basic concepts and well-known results.
\subsection{Viscosity concepts}
Consider the following parabolic complex  Monge-Amp\`ere equations on a bounded domain $\Omega\subset \C^n$
\begin{equation}\label{PMAfree}
e^{\dt u+F(t, z, u)}\mu(t, z)=\MAu,
\end{equation}
where
\begin{itemize}
	\item  $\Omega_T=(0, T)\times\Omega$.
	\item $F(t, z, r)$ is continuous in $[0, T]\times\bar{\Omega}\times\R$ and non-decreasing in $r$.
	\item $\mu(t, z)=f(t, z)dV$,
	 where $dV$ is the standard volume form in $\C^n$ and $f\geq 0$ is a bounded  continuous function in $[0, T]\times\Omega$.
\end{itemize}
\begin{Def} (Test functions)
	Let  $ w : \Omega_T \longrightarrow \R$ be any function defined in $\Omega_T$ and $(t_0,z_0) \in \Omega_T$ a given point. 
	An upper test function (resp. a lower test function) for $w$ at the point $(t_0,z_0)$ 
	is  a $C^{(1,2)}$-smooth function $q$ (i.e., $q$ is $C^1$ in $t$ and $C^2$ in $z$)
	in a neighbourhood of the point $(t_0,z_0)$ such that $  w (t_0,z_0) = q (t_0,z_0)$ and $w \leq q$ (resp. $w \geq q$) in a neighbourhood of $(t_0,z_0)$.
	We will write for short $w \leq_{(t_0,z_0)} q$ (resp.  $w \geq_{(t_0,z_0)} q$).   
\end{Def}

\begin{Def}
	1. A  function  $ u : [0,T) \times \overline{\Omega} \longrightarrow \R$ is said to be a (viscosity) subsolution to the parabolic complex Monge-Amp\`ere equation \eqref{PMAfree} in $(0,T) \times \Omega$ if $u$ is upper semi-continuous in $[0,T) \times \overline{\Omega}$ and for any point $ (t_0,z_0) \in \Omega_T := (0,T) \times \Omega$ and any upper test function $q$ for $u$ at $(t_0,z_0)$, we have
	$$
	(dd^c q_{t_0} (z_0))^n  \geq e^{\partial_t q (t_0,z_0) + F (t_0,z_0,q (t_0,z_0))} \mu (t_0, z_0).
	$$
	In this case, we also say that $u$ satisfies the differential inequality 
	$$(dd^c u)^n \geq e^{\partial_t u (t,z) + F (t,z,u (t,z))} \mu(t, z),$$ in the viscosity sense in $\Omega_T.$
	
	The function $u$ is called  a subsolution to  the Cauchy-Dirichlet problem \eqref{PMA} if $u$ is a subsolution to \eqref{PMAfree}
	satisfying $u\leq\varphi$ in $[0, T)\times\partial\Omega$ and
	$u(0, z)\leq u_0(z)$ for all $z\in\Omega$.
	
	\medskip
	2. A  function  $ v : [0,T) \times \overline{\Omega} \longrightarrow \R$ is said to be a (viscosity) 
	supersolution to the parabolic complex Monge-Amp\`ere equation \eqref{PMAfree}
	in $\Omega_T$ if $v$ is lower semi-continuous in $\Omega_T$ and for any point 
	$ (t_0,z_0) \in \Omega_T$ and any lower test  function $q$ for $v$ at $(t_0,z_0)$ such that $dd^c q_{t_0}(z_0) \geq 0$, we have 
	$$
	(dd^c q_{t_0})^n (z_0) \leq e^{\partial_{t} q (t_0,z_0) + F (t_0,z_0,q (t_0,z_0))} \mu (t_0, z_0).
	$$
	In this case we also say that $v$ satisfies the differential inequality 
	$$(dd^c v)^n \leq e^{\partial_t v (t,z) + F (t,z,v(t,z))} \mu(t, z),$$
	 in the viscosity sense in $\Omega_T$.
	
	The function $v$ is called  a supersolution to \eqref{PMA} if $v$ is a supersolution to \eqref{PMAfree}
	satisfying $v\geq\varphi$ in $[0, T)\times\partial\Omega$ and
	$v(0, z)\geq u_0(z)$ for all $z\in\Omega$.
	
	\medskip
	3. A function $ u: \Omega_T\longrightarrow \R$ is said to be a (viscosity) solution to the parabolic complex Monge-Amp\`ere equation \eqref{PMAfree}
	(respectively, \eqref{PMA}) in $\Omega_T$ if it is a subsolution and a supersolution to
 \eqref{PMAfree} (respectively, \eqref{PMA}) in $\Omega_T$. 
\end{Def}
\begin{Def}
	 A discontinuous viscosity solution to the equation \eqref{PMAfree} (resp.
	 \eqref{PMA}) is a function $u:\Omega_T \to [+\infty, -\infty]$ such that
	
	i) the usc envelope $u^*$ of $u$ satisfies $\forall z\in \Omega, \ u^*(t,z)<+\infty$ and is a viscosity subsolution to \eqref{PMAfree} (resp. \eqref{PMA}),
	
	ii) the lsc envelope $u_*$ of $u$ satisfies $\forall z\in \Omega, \ u_*(t,z)>-\infty$ and is a viscosity supersolution to the equation \eqref{PMAfree} (resp. \eqref{PMA}).
\end{Def}
\subsection{Basic properties}
We recall some basic properties of viscosity subsolution and viscosity supersolution.
\begin{Lem}\label{lem.uplowesti}
Consider the equations
\begin{equation}\label{PMAfree1}
e^{\dt u+F_1(t, z, u)}\mu_1(t, z)=\MAu\qquad\mbox{ in }\qquad (0,T)\times\Omega,
\end{equation}
and
\begin{equation}\label{PMAfree2}
e^{\dt u+F_2(t, z, u)}\mu_2(t, z)=\MAu\qquad\mbox{ in }\qquad (0,T)\times\Omega,
\end{equation}
where, for $j=1, 2$,
\begin{itemize}
	\item $F_j(t, z, r)$ is continuous in
	 $[0, T)\times\bar{\Omega}\times\R$ and non-decreasing in $r$.
	\item $\mu_j(t, z)=f_j(t, z)dV$
	with  $f_j$ is a bounded continuous function in $[0, T)\times\Omega$.
\end{itemize}
Assume that $F_1\geq F_2$ and $\mu_1\geq\mu_2$. If $u_1$ is a subsolution to
\eqref{PMAfree1} then $u_1$ is also a subsolution to \eqref{PMAfree2}.
Conversely, if $u_2$ is a supersolution to \eqref{PMAfree2} then 
 $u_2$ is also a supersolution to \eqref{PMAfree1}.
\end{Lem}
\begin{Lem}\label{lem.changevar}
Let $A>0$. If $u(t, z)$ is a subsolution (resp. supersolution) to
\eqref{PMAfree} in $(0, T)\times\Omega$ then 
$u_A:=\dfrac{1}{A}u(At, z)$ is a subsolution (resp. supersolution) to the equation
\begin{equation}
\dfrac{1}{A^n}e^{\dt u_A+F(At, z, Au_A)}\mu(At, z)=(dd^c u_A)^n,
\end{equation}
in $(0, \dfrac{T}{A})\times\Omega$.
\end{Lem}
 
\begin{Lem}\label{lem.seq}\cite{CIL92, IS13, EGZ15b}
Let $\mu^j (t,x) \geq 0$ be a sequence of continuous volume forms converging uniformly to a volume form $\mu$ on $\Omega_T$ and let $F^j$ be a sequence of continuous functions in $[0,T[ \times \Omega \times \R$ converging locally uniformly to a function $F$.
Let $(u^j)$ be a   locally uniformly bounded sequence of real valued functions defined in $\Omega_T$. 

1. Assume that for every $j \in \N$, $u^j$ is a viscosity subsolution to the complex Monge-Amp\` ere flow
$$
e^{ \partial_t u^j + F^j (t,z,u^j)} \mu^j (t,z) - (dd^c u^j_t)^n = 0,
$$
associated to $(F^j,\mu^j)$ in $\Omega_T$.
Then  its upper relaxed semi-limit 
$$
\overline u= \limsup {}^*_{j \to + \infty} u^j
$$
of the sequence $(u^j)$ is a subsolution to the parabolic Monge-Amp\` ere equation
$$
e^{ \partial_t u + F (t, z, u)} \mu - (dd^c u)^n = 0, 
$$
in $\Omega_T$. 

2. Assume that for every $j \in \N$, $u^j$ is a viscosity supersolution to the complex Monge-Amp\`ere flow associated to $(F^j,\mu^j)$ in $\Omega_T$. Then the lower relaxed semi-limit 
$$
\underline u = {\liminf}{}_{* j \to + \infty} u^j
$$ 
of the sequence  $(u^j)$ is a supersolution to the complex Monge-Amp\`ere flow associated to $(F,\mu)$ in $\Omega_T$.
	
\end{Lem}
One of applications of Lemma \ref{lem.seq} is the following
\begin{Lem}\label{lem modify}
Let $u$  be a subsolution  to the equation
\begin{equation}\label{eq1}
	e^{\dt w+ F_1(t, z, w)}f_1(t, z)dV=(dd^c w)^n,
\end{equation}
and $v$ be  a supersolution to the equation
\begin{equation}\label{eq2}
e^{\dt w+ F_2(t, z, w)}f_2(t, z)dV=(dd^c w)^n,
\end{equation}
in $(0, T)\times\Omega$.
Let $p$ be a negative plurisubharmonic function in $\Omega$ and 
$h: (0, T)\rightarrow	[0, \infty)$ be a continuous non-decreasing function. 
Then
$\tilde{u}(t, z)=u(t, z)+p(z)-h(t)$  is a subsolution to \eqref{eq1} and  
$\tilde{v}(t, z)=v(t, z)-p(z)+h(t)$
is a supersolution to \eqref{eq2}. Moreover, if $p\in\E (\Omega)$ (see Section \ref{subsec Cegrell} for the definition of the class $\E$) 
and there exist $C_1, C_2>0$ such that
\begin{equation}\label{eq3}
	\dt u, \dt v\leq C_1,
\end{equation}
in the viscosity sense and 
\begin{equation}\label{eq4}
	 \sup F(. , . , \sup v), \sup F(., ., \sup u)\leq C_2,
\end{equation}
and
\begin{equation}\label{eq5}
	(dd^c p)^n\geq e^{C_1+C_2}|f_1(t, z)-f_2(t, z)|dV,
\end{equation}
in $\Omega$ for every $t\in (0, T)$ then $\tilde{u}$ is a subsolution to \eqref{eq2}
and $\tilde{v}$ is a supersolution to \eqref{eq1}.
\end{Lem}
\begin{proof}
	Let $B\Subset\Omega$ be a ball and $0<a<b<T$. Then, there exist 
	$h_j: (a, b)\rightarrow [0, \infty)$ and 
	$p_j\in PSH(B)\cap C^{\infty}(\overline{B})$ such that
	\begin{itemize}
		\item $h_j$ is smooth and non-decreasing for every $j\in\N$.
		\item $h_j\searrow h$ in $(a, b)$ and $p_j\searrow p$ in $\overline{B}$ as
		$j\rightarrow\infty$.
	\end{itemize} 
By the definition of viscosity subsolution and viscosity supersolution, we get 
$u(t, z)+p_j(z)-h_j(t)$ is a subsolution to \eqref{eq1} and  
$v(t, z)-p_j(z)+h_j(t)$ is a supersolution to \eqref{eq2} in $(a, b)\times B$ for every $j$.
Hence, by Lemma \ref{lem.seq}, we have $u(t, z)+p(z)-h(t)$ is a subsolution to \eqref{eq1} and  $v(t, z)-p(z)+h(t)$ is a supersolution to \eqref{eq2} in $(a, b)\times B$.
Since $a, b$ and $B$ are arbitrary, we obtain the first conclusion.

Now, we assume that \eqref{eq3}, \eqref{eq4} and \eqref{eq5} are satisfied. Denote
	$$g(z)=\max\limits_{a\leq t\leq b}|f_1(t, z)-f_2(t, z)|.$$
	 Then $g\in C(\Omega)$ (by Lemma \ref{lem continuity}) and
	 $$(dd^c p)_P^n\geq e^{C_1+C_2}g(z)dV,$$
	 in $\Omega$. For every $0<\epsilon< d(B, \partial\Omega)$, we denote
	 $$g_{\epsilon}(z)=\inf\limits_{|\xi-z|<\epsilon}g(\xi).$$
 Then by \cite{EGZ}
 (pages 1064-1066) and by using convolution, the functions $p_j$  can be chosen such that
 \begin{center}
 		$(dd^c p_j)^n\geq e^{C_1+C_2}g_{1/2^j}dV,$
 \end{center}
in $B$ for every $j$. 

Note that $g_{\epsilon}$ converges uniformly to $g$ in $B$ as $\epsilon\searrow 0$. Let $\epsilon_j\searrow 0$ such that 
$$(dd^c\epsilon_j|z|^2)^n\geq e^{C_1+C_2}|g_{1/2^j}(z)-g(z)|dV,$$
in $B$. Denote $q_j=p_j+\epsilon_j|z|^2$, we have 	$q_j\in PSH(B)\cap C^{\infty}(\overline{B})$, $q_j\searrow p$ in $\overline{B}$ as
$j\rightarrow\infty$ and
$$(dd^cq_j)^n\geq e^{C_1+C_2}gdV.$$

Then, by the definition,  $u(t, z)+q_j(z)-h_j(t)$ is a subsolution to \eqref{eq2}
and $v(t, z)-q_j(z)+h_j(t)$ is a supersolution to \eqref{eq1} in $(a, b)\times B$
for every $j$.

 Hence, by Lemma \ref{lem.seq}, we obtain the second conclusion.
\end{proof}
\begin{Lem}\label{lem continuity}
	Assume that $U$ is an open subset of $\R^m\, (m\in\Z^+)$  and $a, b$ are real numbers with $a<b$. 
	Suppose that $g: [a, b]\times U\rightarrow\R$ is a continuous function. Then 
	$h(x)=\sup\limits_{a\leq t\leq b}g(t, x)$ is a continuous function on $U$.
\end{Lem}
\begin{proof}
	Since $\{g(t, x)\}_{a\leq t\leq b}$ is a family of continuous functions on $U$, we get that $h$
	is lower semi-continuous on $U$. It remains to show that $h$ is upper semi-continuous.
	
	Fix $x_0\in U$. Let $V\Subset U$ be an open neighborhood of $x_0$. 
	Then, there exists a sequence $\{x_k\}_{k=1}^{\infty}\subset V\setminus\{x_0\}$ such that
	 $x_k\stackrel{k\to\infty}{\longrightarrow}x_0$ and 
	 $\lim\limits_{k\to\infty}h(x_k)=\limsup\limits_{x\to x_0}h(x)$. 
	 
	 Since $g(t, x_k)$ is continuous on $[a, b]$, there exists $t_k\in [a, b]$ such that $h(x_k)=g(t_k, x_k)$.
	 Since $g$ is continuous on the compact set $[a, b]\times\overline{V}$, we have $g$ is 
	 uniformly continuous on $[a, b]\times\overline{V}$. Then
	 \begin{center}
	 	$\lim\limits_{k\to\infty}(g(t_k, x_k)-g(t_k, x_0))=0$.
	 \end{center}
 Hence
 \begin{center}
 	$\lim\limits_{k\to\infty}(h(x_k)-h(x_0))
 	=\lim\limits_{k\to\infty}(g(t_k, x_k)-h(x_0))
 	\leq\lim\limits_{k\to\infty} (g(t_k, x_k)-g(t_k, x_0))=0$.
 \end{center}
Thus $h(x_0)\geq \limsup\limits_{x\to x_0}h(x)$. Since $x_0$ is arbitrary, we get that $h$ is upper semi-continuous.

The proof is completed.
\end{proof}
\subsection{Comparison principle and Perron envelope}
As is often the case in the viscosity theory and pluripotential theory, one of the  main technical tools is the  comparison principle:
\begin{The}\label{compa.the}
	\cite{EGZ15b}
	Let $u$  (resp. $v$) be a bounded subsolution (resp. supersolution) to the parabolic complex Monge–Amp\`ere equation
	\eqref{PMAfree} in $\Omega_T$.  Assume that one of the following conditions
	is satisfied
	\begin{itemize}
		\item [a)] $\mu (t, z)>0$ for every $(t, z)\in (0, T)\times\Omega$.
		\item[b)] $\mu$ is independent of $t$.
		\item[c)] Either $u$ or $v$ is locally Lipschitz in $t$ uniformly in $z$.
	\end{itemize}
	Then
	\begin{center}
		$\sup\limits_{\Omega_T}(u-v)\leq\sup\limits_{\partial_P(\Omega_T)}(u-v)_+,$
	\end{center}
where $u$ (resp. $v$) has been extended as an upper (resp. a lower) semicontinuous
function to $\overline{\Omega_T}$.
\end{The}
Another main technical tool for viscosity theory is Perron method.
 By using Perron method, one get the following result:
\begin{Lem}\label{lem perron EGZ15b}
	\cite[Lemma 5.4]{EGZ15b}
	Given any non empty family $S_0$ of bounded subsolutions to the parabolic
	equation \eqref{PMAfree} which is bounded above by a continuous function,
	the usc regularization	of the upper envelope 
	$\phi_{S_0}=\sup\limits_{\phi\in S_0}\phi$ is a subsolution to \eqref{PMAfree}.
	
	If $S$ is the family of all subsolutions to the Cauchy-Dirichlet problem 
	\eqref{PMA}, its envelope $\phi_S$  is a
	discontinuous viscosity solution to \eqref{PMAfree}. 
\end{Lem}
The above lemma is useful for the case where $\mu (t, z)>0$ or $\mu$ is independent of $t$. 
In this case, if for every $\epsilon>0$, there exist $\epsilon$-subarriers and 
$\epsilon$-superbarrier (see the below definition) for \eqref{PMA} then by combining 
Lemma \ref{lem perron EGZ15b} and Theorem \ref{compa.the}, one obtain the existence of solution 
to \eqref{PMA}.
\begin{Def}
	a) A function $u\in USC([0, T)\times\bar{\Omega})$ is called $\epsilon$-subbarrier for
	\eqref{PMA} if $u$ is subsolution to \eqref{PMAfree}
	in the viscosity sense such that $u_0-\epsilon\leq u_*\leq u\leq u_0$ in 
	$\{0\}\times\bar{\Omega}$ and $\varphi-\epsilon\leq u_*\leq u\leq \varphi$ in
	$[0, T)\times\partial\Omega.$\\
	b) A function $u\in LSC([0, T)\times\bar{\Omega})$ is called $\epsilon$-superbarrier for
	\eqref{PMA} if $u$ is supersolution to \eqref{PMAfree}
	in the viscosity sense such that $u_0+\epsilon\geq u^*\geq u\geq u_0$ in 
	$\{0\}\times\bar{\Omega}$ and $\varphi+\epsilon\geq u^*\geq u\geq \varphi$ in
	$[0, T)\times\partial\Omega.$
\end{Def}
In this paper, we need a modified version of Lemma \ref{lem perron EGZ15b}:
\begin{Lem}\label{lem perron}
	Assume that for every $\epsilon>0$, the problem \eqref{PMA} admits 
	a continuous $\epsilon$-superbarrier which is Lipschitz in $t$
	and a continuous $\epsilon$-subbarrier. 
	Denote by $S$ the family of all continuous subsolutions  to	\eqref{PMA}.
	Then $\phi_S=\sup\{v: v\in S\}$ is a
	 discontinuous viscosity solution to \eqref{PMA}. 
\end{Lem}
\begin{proof} 
	By the existence of $\epsilon$-subbarriers and $\epsilon$-superbarriers,
	 and by using Theorem \ref{compa.the}, we have
	 $\phi_S$ satisfies the boundary condition and initial condition of \eqref{PMA}.
	  Moreover, the set of
	  continuous points of $\phi_S$ contains $\partial_P\Omega_T$. Hence, $\phi_S$ is a 
	  discontinuous viscosity solution to \eqref{PMA} iff it is a 
	  discontinuous viscosity solution to \eqref{PMAfree}.	By Lemma \ref{lem perron EGZ15b},
	   we have $(\phi_S)^*$ is a subsolution to \eqref{PMAfree}. Then, it remains to show
	 that $(\phi_S)_*(=\phi_S)$ is a supersolution to \eqref{PMAfree}.
	 
	 Assume that $\phi_S$ is not a supersolution to \eqref{PMAfree}. Then, there exist a point
	 $(t_0, z_0)\in (0, T)\times\Omega$, an open neighborhood $I\times U\Subset (0, T)\times\Omega$ of 
	 $(t_0, z_0)$ and a $C^{(1, 2)}$ function $p: I\times U\rightarrow\R$ such that
	 \begin{itemize}
	 	\item $(\phi_S-p)(t_0, z_0)=\min\limits_{I\times U} (\phi_S-p)$;\\
	 	\item At $(t_0, z_0)$, the complex Hessian matrix 
	 	$(\frac{\partial^2p}{\partial z_j\partial\bar{z}_k})_{1\leq j, k\leq n}$ of $p$ is positive and
	 	$(dd^c p)^n>e^{\partial_{t}p(t_0, z_0)+F(t_0, z_0, p(t_0, z_0))}\mu (t_0, z_0)$.
	 \end{itemize}
	Let $0<\delta\ll 1$ such that, in a neighborhood $I_{\delta}\times U_{\delta}\subset I\times U$
	of   $(t_0, z_0)$, the functions $$p_{\delta}^{\pm}=p-\delta |z-z_0|^2\mp\delta (t-t_0)$$
	 are plurisubharmonic in $z$ and
	 $$(dd^c p_{\delta}^\pm)^n>e^{\partial_{t}p_{\delta}^\pm+F(t, z, p_{\delta}^\pm(t, z))}\mu (t, z).$$ Let $0<\delta_1\ll 1$ such that 
	 $$\delta_2:=-\delta_1+\delta\min\limits_{\partial (I_{\delta}\times U_{\delta})}
	 (|t-t_0|+|z-z_0|^2)>0.$$ We define
	 \begin{center}
	 $\tilde{p}:=\min\{p_{\delta}^+, p_{\delta}^-\}+\delta_1=
	 p-\delta(|t-t_0|+|z-z_0|^2)+\delta_1.$
	 \end{center}
	 Then $\tilde{p}$ is a continuous subsolution to \eqref{PMAfree}
	  in $I_{\delta}\times U_{\delta}$
	 such that $\tilde{p}<\phi_S-\delta_2$ on $\partial (I_{\delta}\times U_{\delta})$
	 and  $\tilde{p}(t_0, z_0)>\phi_S (t_0, z_0)$.
	 
	 By the definition of $\phi_S$, for each $a\in \partial(I_{\delta}\times U_{\delta})$,
	  there exists  $u_a\in S$ such that 
	 $$u_a(a)>\phi_S(a)-\frac{\delta_2}{2}>\tilde{p}(a)+\dfrac{\delta_2}{2}.$$
	 By the continuity of $u_a$ and $\tilde{p}$, there exists a neighborhood $V_a$ of $a$
	 such that $u_a>\tilde{p}-\delta_2/2$ in $V_a$. Since 
	 $\partial(I_{\delta}\times U_{\delta})$ is compact, there exists $a_1,..., a_m\in
	 \partial(I_{\delta}\times U_{\delta})$ such that $\partial(I_{\delta}\times U_{\delta})$
	 is covered by $\{V_{a_j}\}_{j=1}^m$. Define
	 \begin{center}
	 $u=\max\{u_{a_1},..., u_{a_m}\}.$
	 \end{center}
	 Then $u\in S$ and $u>\tilde{p}+\dfrac{\delta_2}{2}$ near
	  $\partial(I_{\delta}\times U_{\delta})$. Define
	  \begin{center}
	  $\tilde{u}(t, z)=\begin{cases}
	  u(t,z)\qquad\mbox{if}\quad (t, z)\notin I_{\delta}\times U_{\delta},\\
	  \max\{u(t, z), \tilde{p}(t, z)\}\qquad\mbox{if}\quad (t, z)\in I_{\delta}\times U_{\delta}.
	  \end{cases}$
	  \end{center}
	  Then $\tilde{u}\in S$ and $\tilde{u}(t_0, z_0)\geq \tilde{p}(t_0, z_0)>\phi_S(t_0, z_0)$.
	  We get a contradiction. 
	   
	   Thus $\phi_S$ is a supersolution to \eqref{PMAfree}.
	   
	   The proof is completed.
\end{proof}
Similar to Lemma \ref{lem perron}, we also have
\begin{Lem}\label{lem perron2}
	Assume that for every $\epsilon>0$, the problem \eqref{PMA} admits 
	a continuous $\epsilon$-subbarrier which is Lipschitz in $t$
	and a continuous $\epsilon$-superbarrier. 
	Denote by $S$ the family of all continuous subsolutions  to	\eqref{PMA}
	which is Lipschitz in $t$.
	Then $\phi_S=\sup\{v: v\in S\}$ is a
	discontinuous viscosity solution to \eqref{PMA}. 
\end{Lem}
Using the comparison principle, we have the following   $L^\infty$ a priori estimates to the viscosity solution to \eqref{PMA}.
\begin{Prop} \label{l^infty:apriori}
	Consider the Cauchy-Dirichlet problem \eqref{PMA} (with $\Omega$ is a smooth bounded
	strongly pseudoconvex domain).  Suppose that either $f(t,z)>0$ for every $(t, z)\in (0,T)
\times \Omega$ or $f$ is independent of $t$.  	If $u$ is a solution to \eqref{PMA} then there exists $C>0$ depending on $\Omega$, 
	$\sup\limits_{[0, T)\times\partial\Omega}|\varphi|$, $\min\limits_{\bar{\Omega}}u_0$,
	$\sup\limits_{[0, T)\times\Omega}f$,
	$\sup\limits_{[0, T)\times\bar{\Omega}}F(t, z, \max\varphi)$ such that
	\begin{center}
		$|u|\leq C$,
	\end{center}
	in $[0, T)\times\bar{\Omega}$.
\end{Prop}
\begin{proof}
	Let $\rho\in C^2(\bar{\Omega})\cap PSH(\Omega)$ such that
	$\rho|_{\partial\Omega}=0$ and $(dd^c\rho)^n\geq\mu (t, .)$
	for all $t$.
	We define
$$		\overline{u}=\sup\limits_{[0, T]\times\partial\Omega}\varphi=const,
$$
and
$$	\underline{u}=m+M\rho,
$$ 
where 
$$m=\min\{-\sup\limits_{[0, T)\times\partial\Omega}|\varphi|,
	\min\limits_{\overline{\Omega}}u_0\}$$
and	$$M=\exp(\sup\limits_{[0, T]\times\bar{\Omega}}\frac{F(t, z, \max\varphi)}{n}).$$
Then $\underline{u}$ is a subsolution  and $\overline{u}$
is a supersolution to \eqref{PMA}. Moreover, in $\partial_P(\Omega_T)$,
\begin{center}
	$\underline{u}\leq u\leq \overline{u}.$
\end{center}
By the comparison principle (Theorem \ref{compa.the}), we have
$$\underline{u}\leq u\leq \overline{u},$$
in $\Omega_T$.

Hence, in $[0, T)\times\bar{\Omega}$,
\begin{center}
	$|u|\leq C$,
\end{center}
where $C=\max\{\sup\limits_{[0, T]\times\partial\Omega}\varphi, 
|m|+M\max\limits_{\bar{\Omega}}(-\rho) \}$.
\end{proof}
\subsection{Regularizing in time}
Given  a bounded upper semi-continuous function $u : \Omega_T \longrightarrow \R$,
we consider the upper approximating sequence by Lipschitz functions in $t$,  
$$
u^k (t,x) := \sup \{u (s,x) - k \vert s - t\vert , s \in [0,T)\}, \, \, (t,x) \in\Omega_T.
$$

If $v$ is a bounded lower semi-continuous function, we consider the lower approximating sequence of Lipschitz functions in $t$,
$$
v_k (t,x) := \inf \{v (s,x) + k \vert s - t\vert , s \in [0,T)\}, \, \, (t,x) \in\Omega_T.
$$

\begin{Lem} \label{rglzingint.lem}\cite{EGZ15b}
	For $k \in \R^+$, $u^k$ is an upper semi-continuous function which satisfies the following properties:
	\begin{itemize}
		\item $u (t,z) \leq u^k (t,z) \leq \sup_{\vert s - t\vert \leq A\slash k} u (s,z)$, where $A > 2\, osc_{\Omega_T} u$. 
		\item $\vert u^k (t,x) - u^k (s,x) \vert \leq k \vert s - t\vert$, 
		for $(s,z), (t,z) \in \Omega_T$. 
		\item For all $(t_0,z_0) \in [0,T-A/k] \times \Omega$, there exists $t_0^* \in [0,T)$ such that 
		$$
		\vert t_0^* - t_0\vert \leq A\slash k \text{ and } u^k (t_0,z_0) = u (t_0^*,z_0) - k \vert t_0 - t_0^*\vert.
		$$
	\end{itemize}
	
	Moreover, if $u$ satisfies  
	\begin{equation} \label{eq:subdiffineq}
	e^{\partial_t{u}  + F (t,z, u)} \mu(t, z) \leq (dd^ c u)^ n
	\text{ in }   (0, T)\times\Omega,
	\end{equation} 
	in the viscosity sense then  the function  $u^k$ is a subsolution of 
	$$
	e^{\partial_t{w}   + F_k (t, z, u)} \mu_k(t, z)  -  (dd^ c w)^ n = 0
	\text{ in } (A/k,T-A/k) \times \Omega,
	$$
	where $F_k (t, z, r) := \inf_{\vert s -t\vert \leq A \slash k}(F (s,z,r) + k \vert s - t\vert) $ and $\mu_k(t, z)=\inf_{|s-t|<A/k}\mu (s, z)$.
	The dual statement is true for a lower semi-continuous function $v$ which is a supersolution.
\end{Lem}
\subsection{Cegrell's classes}\label{subsec Cegrell}
\medskip
	Let $\Omega$ be a bounded hyperconvex domain in $\C^n$. The following classes of
	plurisubharmonic functions were introduced by Cegrell \cite{Ceg98, Ceg04}:
	\begin{itemize}
		\item $\mathcal{E}_0(\Omega)$ is the set of  bounded psh function $u$ with $\lim_{z\rightarrow \xi} u(z)= 0, \forall \xi\in \partial \Omega$ and $\int_\Omega (dd^cu)_P^n<+\infty$.
		\item $\E (\Omega)$  is the set of all  $u\in PSH^-(\Omega)$ such that for every
		$z_0\in\Omega$, there exist a neighborhood $U$ of $z_0$ in $\Omega$ and a decreasing
		sequence $h_j\in\E_0(\Omega)$ such that $h_j\searrow u$ on $U$ and
		$\sup_j\int_{\Omega}(dd^ch_j)_P^n<\infty$.
		\item $\mathcal{F}(\Omega)$  is the set of all  $u\in PSH^-(\Omega)$ such that there exists a decreasing sequence $u_j\in \mathcal{E}_0(\Omega)$  such that $u_j\searrow u$ on $\Omega$ and $\sup_j \int_\Omega (dd^c u_j)_P^n<+\infty$.
	\end{itemize}
By \cite{Ceg04, Blo06}, $u\in\E (\Omega)$ iff $u$ is a non-positive psh function satisfying
the following property: there exists a Borel measure $\nu$ such that, if $U\subset\Omega$
and $u_j$ is a sequence of bounded psh functions in $U$ satisfying $u_j\searrow u$ then
$(dd^c u_j)_P^n$ converges weakly to $\nu$ in $U$. In this case, the Monge-Amp\`ere operator
of $u$ is defined by $(dd^c u)_P^n:=\nu$.

The class $\F (\Omega)$ satisfies the following property: For every $u\in\F (\Omega)$,
for each $z\in\partial\Omega$,
\begin{center}
	$\limsup\limits_{\Omega\ni\xi\to z}u(\xi)=0$.
\end{center}
Moreover, by \cite{NP09}, 
the comparison principle holds in the class $\F^a(\Omega)=\{u\in\F (\Omega) :
(dd^c u)_P^n$ vanishes on all pluripolar sets $\}$. 

The class $\F (\Omega)$ has been generalized as follows
\begin{Def}\label{def ceg class}
	Let $\Omega$ be a strongly pseudoconvex domain in $\C^n$. Let $\psi\in C(\partial\Omega)$.
	 Then the class $\F (\Omega, \psi)$ is defined by
	\begin{center}
		$\F(\Omega, \psi)=\{u\in PSH (\Omega): \exists v\in \F(\Omega)$
		such that $U_{\psi}\geq u\geq v+U_{\psi} \},$
	\end{center}
where $U_{\psi}$ is the unique solution to the problem 
\begin{equation}
\begin{cases}
U_{\psi}\in C(\overline{\Omega}\cap PSH(\Omega)),\\
(dd^cU_{\psi})_P^n=0,\\
U_{\psi}|_{\partial\Omega}=\psi.
\end{cases}
\end{equation}
\end{Def}
The class $\F (\Omega, \psi)$ has been used to charaterize the boundary behavior in the
Dirichlet problem for Monge-Amp\`ere equation.
\begin{The}\label{the MA Cegrell}\cite{Ceg04, Aha07}
	Let $\Omega$ be a strongly pseudoconvex domain in $\C^n$. Let $\nu$ be a positive Borel
	measure in $\Omega$ and $\psi\in C(\partial\Omega)$. If $\nu (\Omega)<\infty$ and
	$\nu$ vanishes on all pluripolar sets then there exists a unique function
	$u\in \F(\Omega, \psi)$ such that $(dd^cu)_P^n=\nu$.
\end{The}
\section{Local regularity in time}
In this section, we assume that $\Omega$ is a bounded domain in $\C^n$.
 We will prove some results on
the local regularity in time of solution to \eqref{PMA} by using the following comparison
principle.
\begin{The}\label{the weak compa}
	Let $u$ and $v$ be, respectively, a bounded subsolution and
	a bounded supersolution to \eqref{PMA}.
	\begin{itemize}
		\item [a)]  Assume	that for every $K\Subset\Omega$,
		for every $0<R<S<T$ and $\epsilon>0$, there exists $0<\delta\ll 1$
		such that	$(1+\epsilon)f(t+s, z)\geq f(t, z)$ for all $z\in K$, $0<s<\delta$
		and $R<t<S$. Then, for every  $0<R<S<T$, for every $\epsilon>0$, 
		there exists $0<\delta\ll 1$ such that
		\begin{center}
			$u(t+s, z)<v(t, z)+\epsilon,$
		\end{center}
		for every $(t, z)\in (R, S)\times\Omega$ and $s\in (0, \delta)$.
		In particular, if either $u$ or $v$ is continuous in $t$ then $u\leq v$.
		\item[b)] Assume that for every $K\Subset\Omega$,
		for every $0<R<S<T$ and $\epsilon>0$, there exists $0<\delta\ll 1$
		such that	$(1+\epsilon)f(t, z)\geq f(t+s, z)$ for all $z\in K$, $0<s<\delta$
		and $R<t<S$. Then, for every  $0<R<S<T$, for every $\epsilon>0$, 
		there exists $0<\delta\ll 1$ such that
		\begin{center}
			$u(t-s, z)<v(t, z)+\epsilon,$
		\end{center}
		for every $(t, z)\in (R, S)\times\Omega$ and $s\in (0, \delta)$.
		In particular, if either $u$ or $v$ is continuous in $t$ then $u\leq v$.
	\end{itemize}
\end{The}
\begin{proof}
	We will prove the part a). The proof of the part b) is similar.
	
	Let $\epsilon>0$ and $0<R<S<T$. By the semi-continuity of $u, v$ and by $u\leq v$ in $\partial_P(\Omega_T)$,
	 there exists $\min\{ R, T-S\}\gg\delta_1>0$ such that
	\begin{equation}\label{eq1 proof w compa}
	u(t, z)\leq v( t, z)+\epsilon,
	\end{equation}
	for every $(t, z)\in ([0, 2\delta_1]\times\Omega)\cup ([0, S+\delta_1]\times (\overline{\Omega}\setminus\Omega_{\delta_1}))$, where
	\begin{center}
		$\Omega_{\delta_1}=\{z\in\Omega: dist (z, \partial\Omega)<\delta_1 \}$.
	\end{center}

	By the assumption, there exists 
	$\delta_2\in (0, \delta_1)$ such that
\begin{equation}\label{ineq_vol_F}
(1+\epsilon)f(t_1, z)\geq f(t_2, z) \, \, \text{and }|F(t_1, z, r)-F(t_2, z, r)|<\epsilon
\end{equation}
for every $z\in\Omega_{\delta_1}$, $r\in [-M, M]$, $t_1, t_2\in (\delta_1, S+\delta_1)$ with
	$t_2<t_1<t_2+\delta_2$. Here $M=\sup\limits_{[0,T)\times\overline{\Omega}}|u|$.
	
	Denote, for every $(t, z)\in [0, T)\times\overline{\Omega}$,
	\begin{center}
		$u^k(t, z)=\sup\{u(s, z)-k|t-s|: s\in [0, T) \},$
	\end{center}
then $u^k$ is Lipschitz in $t$. It follows from Lemma  \ref{rglzingint.lem} and \eqref{ineq_vol_F} that if $0<\delta<\delta_2/2$ and
	$k>\dfrac{(A+1)}{\delta}$, for some $A>2 osc_{\Omega_T}u$  then
	$u^k(t+\delta, z)-t\log (1+\epsilon)-(1+t)\epsilon$
	is a subsolution to 
	\begin{equation}
	e^{\dt w+F(t, z, w)}\mu(t, z)=(dd^cw)^n,
	\end{equation}
	in $(\delta_1, S)\times\Omega_{\delta_1}.$
 By using \eqref{eq1 proof w compa} and 
 the comparison principle (Theorem \ref{compa.the}), we get
	\begin{center}
		$ u^k(t+\delta, z)-t\log (1+\epsilon)-(1+t)\epsilon\leq v(t, z)$,
	\end{center}
	for every $(t, z)\in(\delta_1, S)\times\Omega.$
	
\medskip	
	Since $u^k\geq u$ and $0<\log (1+\epsilon)<\epsilon$, we have
	\begin{equation}\label{eq2 proof w compa}
		 u(t+\delta, z)-(1+2T)\epsilon\leq v(t, z),
	\end{equation}
	for every $(t, z)\in(\delta_1, S)\times\Omega_{\delta_1}.$
	
	Combining \eqref{eq1 proof w compa} and \eqref{eq2 proof w compa}, we obtain
	\begin{center}
		$u(t+s, z)<v(t, z)+(1+2T)\epsilon,$
	\end{center}
	for every $(t, z)\in (R, S)\times\Omega$ and $0<s<\delta$.
	
	 The proof is completed.
\end{proof}
In Theorem \ref{the weak compa}, if we assume that $u$ and $v$ are continuous in $t$ then
we have $u\leq v$.
As a consequence, we have the following results on the Lipschitz regularity in time of viscosity solutions. This kind of regularity is necessary to define parabolic pluripotential solutions (cf. \cite{GLZ1,GLZ2}).
\begin{Prop}\label{rgltionint1.prop}
	Assume that  $\mu$  is non-increasing in
	$t$	and $u$ is a solution to \eqref{PMA}. Suppose that
	there exists $C_0>0$ satisfying
	\begin{center}
		$\varphi (t, z)-\varphi (s, z)\geq -C_0(t-s), \quad\forall z\in\partial\Omega,
		0<s<t<T,$
	\end{center}
	and for every $m>0$, there exists $C_m>0$ satisfying
	\begin{center}
		$F(t, z, r)-F(s, z, r)\leq C_m(t-s), \quad\forall r\in [-m, m], z\in\partial\Omega,
		0<s<t<T.$
	\end{center}
	Denote $M=\sup |u|$, $N=\sup |\varphi|$. Then, for every $ 0<B<A<T$,
	\begin{equation}\label{eq rgltionint1.prop}
		\dfrac{u(B, z)-u(A, z)}{A-B}\leq \dfrac{2M}{A}+\max\{C_0,
		BC_M\}+n+N.B,	
	\end{equation}
	for all $z\in\bar{\Omega}$.
	
	In particular, $\dt u\geq -\dfrac{2M}{t}-\max\{C_0, tC_M\}-n-Nt$ in the viscosity sense.
\end{Prop}
\begin{proof}
	The idea of the proof is similar to Theorem 4.2  in \cite{GLZ3}.
	
	We consider $u_A=\dfrac{1}{A}u(At, z)$ and $u_B=\dfrac{1}{B}u(Bt, z)$ in $[0, 1]\times\bar{\Omega}$.
	
	 By Lemma \ref{lem.changevar}, in $(0, 1)\times\Omega$, we have
	\begin{center}
		$(dd^cu_A)^n=\dfrac{1}{A^n}e^{\partial_tu_A+F(At, z, Au_A)}\mu(At, z),$
	\end{center}
	and
	\begin{center}
		$
		(dd^cu_B)^n =\dfrac{1}{B^n}e^{\partial_tu_B+F(Bt, z, Bu_B)}\mu(Bt, z)
		=\dfrac{1}{A^n}e^{\partial_tu_B+F(Bt, z, Bu_B)+n\log (B/A)}\mu (Bt, z),
		$
	\end{center}
	 in the viscosity sense.
	
	By the assumption, we have, for every $(t, z)\in (0, 1)\times\Omega$,
	\begin{center}
		$\mu (Bt, z)\geq \mu (At, z),$
	\end{center}
	and
	\begin{flushleft}
		$\begin{array}{ll}
		F(Bt, z, Bu_B)+n\log (B/A)&\geq F(At, z, Bu_B)-C_M(A-B)t-\dfrac{n(A-B)}{B}\\
		&=F(At, z, Au_B-(A-B)u_B)-(C_M+\dfrac{n}{B})(A-B)\\
		&\geq F(At, z, Au_B-(A-B)\sup u_B)-(C_M+\dfrac{n}{B})(A-B)\\
		&\geq F(At, z, Au_B-N(A-B))-(C_M+\dfrac{n}{B})(A-B).
		\end{array}$
	\end{flushleft}
	Denote
	\begin{center}
		$\tilde{u}_B=u_B-(\max\{\dfrac{C_0}{B},
		C_M\}+\dfrac{n}{B})(A-B)t-(N+\dfrac{M}{AB})(A-B).$
	\end{center}
	We have, by Lemma \ref{lem.uplowesti},
	\begin{center}
		$(dd^c\tilde{u}_B)^n\geq
		\dfrac{1}{A^n}e^{\partial_t\tilde{u}_B+F(At, z, A\tilde{u}_B)}\mu (At, z),$
	\end{center}
	in the viscosity sense in $(0, 1)\times\Omega$. Note that $\tilde{u}_B\leq u_A$ in
	$\partial_P([0, 1)\times\Omega)$. Then, by Theorem \ref{the weak compa}, 
	$\tilde{u}_B\leq u_A$
	in $[0, 1]\times\bar{\Omega}$. In particular, for every $z\in\bar{\Omega}$,
	\begin{center}
		$\dfrac{1}{A}u(A, z)\geq \dfrac{1}{B}u(B, z)
		-(\max\{\dfrac{C_0}{B},
		C_M\}+\dfrac{n}{B}+N+\dfrac{M}{AB})(A-B)$.
	\end{center}
	
	Hence,
	\begin{center}
		$\dfrac{u(B, z)-u(A, z)}{A-B}\leq \dfrac{2M}{A}+\max\{C_0,
		BC_M\}+n+NB$,	
	\end{center}
	for every $z\in\bar{\Omega}$. 
\end{proof}
By the same argument, we have 
\begin{Prop}\label{rgltionint2.prop}
	Assume that  $\mu$ is non-decreasing in
	$t$	and $u$ is a solution to \eqref{PMA}. Suppose that
	there exists $C_0>0$ satisfying
	\begin{equation}
		\varphi (t, z)-\varphi (s, z)\leq C_0(t-s), \quad\forall z\in\partial\Omega,
		0<s<t<T,
	\end{equation}
	and for every $m>0$, there exists $C_m>0$ satisfying
	\begin{center}
		$F(t, z, r)-F(s, z, r)\geq -C_m(t-s), \quad\forall r\in [-m, m], z\in\partial\Omega,
		0<s<t<T.$
	\end{center}
	Denote $M=\sup |u|$. Then, for every $0<B<A<T$, 
	\begin{equation}\label{eq rgltionint2.prop}
\dfrac{u(A, z)-u(B, z)}{A-B}\leq \dfrac{2M}{A}+\max\{C_0,
		BC_M\}+n+M.B,	
\end{equation}
	for all $z\in\bar{\Omega}$.
	
	In particular, $\dt u\leq \dfrac{2M}{t}+\max\{C_0, tC_M\}+n+Mt$ in the viscosity sense.
\end{Prop}
Here, the difference between the right-hand sides of \eqref{eq rgltionint1.prop}
 and \eqref{eq rgltionint2.prop}
 is due to the difference between the upper and lower estimates for $F(At, z, Au_B-(A-B)u_B)$.
\begin{Cor}
	Assume that $\mu, F, \varphi$ satisfy the conditions in Proposition
	 \ref{rgltionint2.prop}. If there exists
	 an open set $U\subset\Omega$ such that
	 $u_0$ is not a maximal plurisubharmonic function in $U$ and
	$\lim_{t\to 0^+}t\log\sup_{z\in U}f(t, z)=-\infty$
	 then \eqref{PMA} does not admit a solution.
\end{Cor}
Combining Proposition \ref{rgltionint1.prop} and Proposition \ref{rgltionint2.prop}, we have
\begin{Cor}\label{rgltionint.cor}
	Assume that  $\mu$ is independent of $t$
	and $u$ is a solution to \eqref{PMA}. Suppose that
	there exists $C_0>0$ satisfying
	\begin{center}
		$|\varphi (t, z)-\varphi (s, z)|\leq C_0|t-s|, \quad\forall z\in\partial\Omega,
		s,t\in [0, T),$
	\end{center}
	and for every $m>0$, there exists $C_m>0$ satisfying
	\begin{center}
		$|F(t, z, r)-F(s, z, r)|\leq C_m|t-s|, \quad\forall r\in [-m, m], z\in\partial\Omega,
		s,t \in [0, T).$
	\end{center}
	Denote $M=\sup |u|$. Then, for every $0<B<A<T$, 
	\begin{center}
		$\dfrac{|u(A, z)-u(B, z)|}{A-B}\leq \dfrac{2M}{A}+\max\{C_0,
		BC_M\}+n+M.B$,	
	\end{center}
	for all $z\in\bar{\Omega}$.
	
	In particular, $|\dt u|\leq \dfrac{2M}{t}+\max\{C_0, tC_M\}+n+Mt$ in the viscosity sense.
\end{Cor}
We then have the following corollary. 
\begin{Cor}\label{cor exist admi}
	Assume that $\mu$, $F$ and $\varphi$ are independent of $t$. If $\eqref{PMA}$ admits
	a viscosity solution then for every $0<t<T$, there exists $C_t>0$ such that
	$(dd^c u (t, z))^n\leq C_t\mu (z)$ in the viscosity sense in $\Omega$. In particular,
	$(u_0, \mu)$ is admissible.	
\end{Cor}

\section{The existence of solution}
In this section, we prove Theorem \ref{mainexistence} and Theorem \ref{mainadmissible}.
 Assume that $\Omega\subset \C^n$ is  a smooth bounded strongly pseudoconvex domain. Consider the  Cauchy-Dirichlet problem spelt out in the Introduction
\begin{equation}\label{PMA_sol}
\begin{cases}
e^{\dt u+F(t, z, u)}\mu(t, z)=\MAu\qquad\mbox{ in }\qquad \Omega_T,\\
u=\varphi\qquad\mbox{in}\qquad [0, T)\times\partial\Omega,\\
u(0, z)=u_0(z)\qquad\mbox{in}\quad\bar{\Omega}. 
\end{cases}
\end{equation}
\subsection{The construction of $\epsilon$-subbarriers and $\epsilon$-superbarriers}
The $\epsilon$-subbarriers and $\epsilon$-superbarriers are important for proving the existence of solutions to the problem \eqref{PMA_sol}. In \cite{EGZ15b}, Eyssidieux-Guedj-Zeriahi have given a method to construct
$\epsilon$-subbarriers and $\epsilon$-superbarriers in the case where $\phi$ and $f$ are independent of $t$
(see \cite[Proposition 5.9]{EGZ15b}).
In this section, we introduce several results about the construction of $\epsilon$-subbarriers and $\epsilon$-superbarriers in the general case.
\begin{Prop}\label{prop.subbarrier}
	For all $\epsilon>0$,
	there exists a continuous $\epsilon$-subbarrier for \eqref{PMA}
	 which is Lipschitz in $t$. 
\end{Prop}
\begin{proof}
	Let $\rho\in C^2(\bar{\Omega})\cap PSH(\Omega)$ such that $\rho|_{\partial\Omega}=0$,
	$\nabla\rho|_{\partial\Omega}\neq 0$
	and $dd^c\rho\geq dd^c |z|^2$. Denote $c=\sup\limits_{\Omega}(-\rho)$. Then, there exists
	$M_1\gg 1$ such that the function
	\begin{center}
		$\underline{u}_1=u_0+\dfrac{\epsilon (\rho-c)}{2c}-M_1t,$
	\end{center}
is a subsolution to \eqref{PMA_sol} satisfying $\underline{u}_1\leq \varphi$ in 
$[0, T]\times\partial\Omega$.

Let $\varphi_{\epsilon}\in C^{\infty}(\R\times\C^n)$ such that
\begin{center}
	$\varphi-\dfrac{\epsilon}{2} \leq\varphi_{\epsilon}\leq\varphi,$
\end{center}
in $[0, T]\times\partial\Omega$. Then, there exists $M_2\gg 1$ such that the function
\begin{center}
	$\underline{u}_2=\varphi_{\epsilon}-\dfrac{\epsilon}{2}+M_2\rho$,
\end{center}
is a subsolution to \eqref{PMA_sol} satisfying $\underline{u}_2\leq u_0$ in 
$\{0\}\times\bar{\Omega}$.

Now, we define $\underline{u}=\max\{\underline{u}_1, \underline{u}_2\}$. It is clear that
$\underline{u}$ is a continuous $\epsilon$-subbarrier for \eqref{PMA_sol}.
	\end{proof}
\begin{Lem}\label{lem short long}
	\begin{itemize}
		\item [1)] Let $0<\epsilon_0<T$.
		Let $u$ be a bounded continuous subsolution to \eqref{PMA_sol} in
		$[0, \epsilon_0)\times\overline{\Omega}$. Then, for every $0<\epsilon<\epsilon_0$,
		there exists a continuous subsolution $\tilde{u}$ to \eqref{PMA_sol} on $[0, T)\times\Omega$ such that
		\begin{center}
			$\begin{cases}
			\tilde{u}\leq u \mbox{ in } [0, \epsilon_0)\times\overline{\Omega},\\
			\tilde{u}= u \mbox{ in } [0, \epsilon)\times\overline{\Omega}.
			\end{cases}$
		\end{center}
		Moreover, if $u$ is Lipschitz in $t$ then $\tilde{u}$ is also 
		Lipschitz in $t$.
	\item [2)] Let $0<\epsilon_0<T$.
	Let $v$ be a bounded continuous supersolution to \eqref{PMA_sol} in
	$[0, \epsilon_0)\times\overline{\Omega}$. Then, for every $0<\epsilon<\epsilon_0$,
	there exists a continuous supersolution $\tilde{v}$ to \eqref{PMA_sol} on $[0, T)\times\Omega$  such that
\begin{center}
	$\begin{cases}
	\tilde{v}\geq v \mbox{ in } [0, \epsilon_0)\times\overline{\Omega},\\
	\tilde{v}= v \mbox{ in } [0, \epsilon)\times\overline{\Omega}.
	\end{cases}$
\end{center}
	Moreover, if $v$ is Lipschitz in $t$ then $\tilde{v}$ is also 
	Lipschitz in $t$.
	\end{itemize}
\end{Lem}
\begin{proof}
1)	Denote
	\begin{center}
		$M=\sup\limits_{[0, \epsilon_0)\times\overline{\Omega}}u$,
	$m=\min\{\inf\limits_{[0, \epsilon_0)\times\overline{\Omega}}u(t, z), 
	\min\limits_{[0, T]\times\partial\Omega}\varphi (t, z) \}$,
\end{center}
and
\begin{center}
	$M_F=\max_{[0, T]\times\overline{\Omega}}F(t, z, m)$, 
	 $M_f=\sup_{[0, T]\times\Omega}f(t, z)$.
\end{center}
	Let $\rho\in C^2(\overline{\Omega})\cap PSH^{-}(\Omega)$ such that
$(dd^c\rho)^n\geq e^{M_F}M_fdV$. We define
\begin{center}
	$h(t)=
	\begin{cases}
	0\qquad\mbox{ for }t<\epsilon,\\
	C(t-\epsilon)\qquad\mbox{ for }t\geq\epsilon,
	\end{cases}$
\end{center}
where $C=1+\dfrac{M-m+\max (-\rho)}{\epsilon_0-\epsilon}$.

Then 
\begin{center}
	$\tilde{u}=
	\begin{cases}
	\max\{u(t, z)-h(t), m+\rho\}\qquad\mbox{ in }
	\quad [0, \epsilon_0)\times\overline{\Omega},\\
	m+\rho \qquad\mbox{ in }
	\quad [\epsilon_0, T)\times\overline{\Omega},
	\end{cases}$
\end{center}
 is a continuous subsolution to \eqref{PMA} satisfying
\begin{center}
	$\begin{cases}
	\tilde{u}\leq u \mbox{ in } [0, \epsilon_0)\times\overline{\Omega},\\
	\tilde{u}= u \mbox{ in } [0, \epsilon)\times\overline{\Omega}.
	\end{cases}$
\end{center}
2) Denote
\begin{center}
	$m=\inf\limits_{[0, \epsilon_0)\times\overline{\Omega}}v$ and
	$M=\max\{\sup\limits_{[0, \epsilon_0)\times\overline{\Omega}}v, 
	\max\limits_{[0, T]\times\partial\Omega}\varphi  \}$.
\end{center}
 We define
\begin{center}
	$h(t)=
	\begin{cases}
	0\qquad\mbox{ for }t<\epsilon,\\
	C(t-\epsilon)\qquad\mbox{ for }t\geq\epsilon,
	\end{cases}$
\end{center}
where $C=1+\dfrac{M-m}{\epsilon_0-\epsilon}$.

Then 
\begin{center}
	$\tilde{v}=
	\begin{cases}
	\min\{v(t, z)+h(t), M\}\qquad\mbox{ in }
	\quad [0, \epsilon_0)\times\overline{\Omega},\\
	M \qquad\mbox{ in }
	\quad [\epsilon_0, T)\times\overline{\Omega},
	\end{cases}$
\end{center}
is a continuous supersolution to \eqref{PMA} satisfying
\begin{center}
	$\begin{cases}
	\tilde{v}\geq v \mbox{ in } [0, \epsilon_0)\times\overline{\Omega},\\
	\tilde{v}= v \mbox{ in } [0, \epsilon)\times\overline{\Omega}.
	\end{cases}$
\end{center}
\end{proof}
\begin{Prop}\label{prop.superbarrier}
	If $(u_0 (z), \mu(0, z))$ is  
	admissible then for all $\epsilon>0$,
	there exists a continuous $\epsilon$-superbarrier for \eqref{PMA_sol}
	which is Lipschitz in $t$. 
\end{Prop}
In order to prove Proposition \ref{prop.superbarrier}, we need the following lemma
\begin{Lem}\label{lem GKZ08} For every $p>1$, 
	there exists a continuous function $A_p: (0, \infty)\rightarrow (0, \infty)$
	with $\lim\limits_{t\to 0^+}A_p(t)=0$ such that: for every $g\in L^p(\Omega)$,
	the unique solution $u$ to the problem
	\begin{equation}
	\begin{cases}
	u\in C(\overline{\Omega})\cap PSH(\Omega),\\
	(dd^cu)^n=gdV,\\
	u|_{\partial\Omega}=0,
	\end{cases}
	\end{equation}
	satisfies $\sup_{\Omega}|u|\leq A_p(\|g\|_{L^p(\Omega)})$.
\end{Lem}
Lemma \ref{lem GKZ08} is an immediate corollary of the following theorem
\begin{The}\label{the GKZ08}\cite[Theorem 1.1]{GKZ}
	Fix $0\leq g\in L^p(\Omega)$, $p>1$. Let $\psi_1, \psi_2$ be two bounded psh functions in $\Omega$ such that
	$(dd^c\psi_1)^n=gdV$, and $\liminf_{z\to\partial\Omega}(\psi_1-\psi_2)(z)\geq 0$. Fix $r\geq 1$ and
	$0\leq\gamma<r/[nq+r]$, $1/p+1/q=1$. Then
	\begin{center}
		$\sup (\psi_2-\psi_1)\leq C\|\max\{\psi_2-\psi_1, 0\}\|_{L^r(\Omega)}^{\gamma},$
	\end{center}
	for some uniform constant $C=C(\gamma, \|g\|_{L^p(\Omega)})$. Moreover, for each $0\leq\gamma<r/[nq+r]$,
	\begin{center}
		$\lim\limits_{\|g\|_{L^p(\Omega)}\to 0^+}C(\gamma, \|g\|_{L^p(\Omega)})=0.$
	\end{center}
\end{The}
The last assertion of Theorem \ref{the GKZ08} is not mentioned in \cite[Theorem 1.1]{GKZ} but it is deduced
from the proof of this theorem (see \cite[pages 1071-1073]{GKZ}).
\begin{proof}[Proof of Proposition \ref{prop.superbarrier}]
	Since $(u_0(z), \mu(0, z))$ is  admissible, there	exist $u_{\epsilon}\in C(\bar{\Omega})$ and $C_{\epsilon}>0$ such that $u_0+\dfrac{\epsilon}{4}\leq u_{\epsilon}\leq u_0+\dfrac{\epsilon}{2}$ and 
	$(dd^c u_{\epsilon})^n\leq e^{C_{\epsilon}}\mu (0, z)$ in the viscosity sense. 
	Denote
	$$M_1=\sup\{|F(t, z, u_0(z))|: (t, z)\in (0, T)\times\Omega \}.$$
	By using the definition,
	 we have $\hat{u}(t, z):=u_{\epsilon}(z)+(C_{\epsilon}+M_1)t$ is a viscosity supersolution to the equation
	\begin{equation}\label{eq1 proof propsuper}
	(dd^cw)^n=e^{\partial_tw+F(t, z, w)}\mu(0, z),
	\end{equation}
	in $(0, T)\times\Omega$.
	
	 For every $0<\delta<T$, we denote
	\begin{center}
		$f_{\delta}(z)=\sup\{|f(t, z)-f(0, z)|: 0\leq t\leq \delta \}.$
	\end{center}
	Then $f_{\delta}\in C(\Omega)\cap L^{\infty}(\Omega)$ (by Lemma \ref{lem continuity})
	 and $f_{\delta}\searrow 0$ as $\delta\searrow 0$. Hence
	\begin{center}
		$\lim\limits_{\delta\to 0^+}\|f_{\delta}\|_{L^2(\Omega)}=0$.
	\end{center}
	Therefore, we can choose $0<\delta_1<T$ such that 
	$A_2(e^{C_{\epsilon}+M_1+M_2}\|f_{\delta_1}\|_{L^2(\Omega)})<\epsilon/2$, where 
	$A_2$ is defined as in Lemma \ref{lem GKZ08} and
	\begin{center}
$M_2:=\sup\{ F(t, z, \sup\limits_{[0, T]\times\partial\Omega}\varphi): (t, z)\in [0, T]\times\overline{\Omega}\}$.
	\end{center}

	Define by $\rho$ the unique solution to the problem
	\begin{equation}
	\begin{cases}
	\rho\in PSH(\Omega)\cap C(\overline{\Omega}),\\
	(dd^c\rho)^n=e^{C_{\epsilon}+M_1+M_2}f_{\delta_1}dV,\\
	\rho|_{\partial\Omega}=0.
	\end{cases}
	\end{equation}
	Then, by using Lemma \ref{lem modify}, we have $\hat{u}-\rho$
	 is a supersolution to the equation
	 \begin{equation}
	 (dd^cw)^n=e^{\partial_tw+F(t, z, w)}\mu(t, z),
	 \end{equation}
	 in $(0, \delta_1)\times\Omega$. Moreover, by Theorem \ref{the GKZ08}, we have $0\leq -\rho\leq \epsilon/2$.
	  
	 Choose $0\leq\delta_2\leq\delta_1$ such that $\hat{u}(t, z)\geq\varphi (t, z)$ 
	 for every $(t, z)\in (0, \delta_2)\times\partial\Omega$. Then $\hat{u}-\rho$ is a supersolution to \eqref{PMA_sol}
	 in $[0, \delta_2)\times\overline{\Omega}$
	 which is Lipschitz in $t$. Moreover $u_0(z)\leq \hat{u}(0, z)-\rho(z)\leq u_0(z)+\epsilon$ in $\Omega$.
	 By using Lemma \ref{lem short long}, there exists a supersolution $\underline{u}_1$ to \eqref{PMA_sol}
	 in $[0, T)\times\overline{\Omega}$  which is Lipschitz in $t$ and satisfies 
	 $u_0(z)\leq \underline{u}_1(0, z)\leq u_0(z)+\epsilon$ in $\Omega$.
	
	Let $\varphi^{\epsilon}\in C^{\infty}(\R\times\C^n)$ such that
	\begin{center}
		$\varphi\leq\varphi^{\epsilon}\leq\varphi+\epsilon,$
	\end{center}
	in $[0, T]\times\partial\Omega$.
	For every $t\in [0, T]$, we denote by $\overline{u}_2(t, z)$ the unique solution to the equation
	\begin{equation}
	\begin{cases}
	\overline{u}_2(t, \cdot)\in PSH(\Omega)\cap C(\overline{\Omega}),\\
	(dd^c\overline{u}_2(t, z))_P^n=0,\\
	\overline{u}_2(t, z)|_{\partial\Omega}=\varphi^{\epsilon}(t, z)|_{\partial\Omega}.
	\end{cases}
	\end{equation}
	By \cite[Theorem A]{BT1} we have
	\begin{center}
		$|\overline{u}_2(t_1, z)-\overline{u}_2(t_2, z)|
		\leq \sup\{|\varphi^{\epsilon}(t, \xi)-\varphi^{\epsilon}(t, \xi)|:\xi\in\partial\Omega \}$,
	\end{center}
for every $t_1, t_2\in [0, T]$, for all $z\in\Omega$.

	Then $\overline{u}_2: [0, T]\times\overline{\Omega}\rightarrow\R$ is Lipschitz in $t$. In particular,
	$\overline{u}_2\in C([0, T]\times\overline{\Omega})$. It is easy to see that $\overline{u}_2$ is
	a supersolution to \eqref{PMA_sol}.
	
	Now, we define $\overline{u}=\min\{\overline{u}_1, \overline{u}_2\}$. It is clear that
	$\overline{u}$ is a continuous $\epsilon$-superbarrier for \eqref{PMA_sol}.
\end{proof}
\begin{Rem}\label{rem_not_admit}
	The converse statement of Proposition \ref{prop.superbarrier} is false.
	For example, if $\Omega$ is the unit ball, $u_0=|z|^2-1$, $\varphi=0$, $F=0$, $\mu=tdV$
	then  for every $T>0$, 
	\begin{center}
		$\overline{u}_T(t, z)=\min\{0, u_0-t\log t+e^Tt\},$
	\end{center}
	is	an $\epsilon$-superbarrier for \eqref{PMA_sol} in $[0, T)\times\bar{\Omega}$ for
	every $\epsilon>0$. But $(u_0, \mu (0, z))=(|z|^2-1, 0)$ is not admissible.
\end{Rem}
\begin{Prop}\label{prop short long}
	Let $\epsilon>0$. If there exists a continuous $\epsilon$-superbarrier $u$ for 
	\eqref{PMA} in $[0, S)\times\overline{\Omega}$ for some $0<S<T$, then
	there exists a continuous $\epsilon$-superbarrier $\tilde{u}$ for 
	\eqref{PMA} in $[0, T)\times\overline{\Omega}$.
	 Moreover, if $u$ is Lipschitz in $t$  then $\tilde{u}$ is also 
	Lipschitz in $t$.
\end{Prop}
\begin{proof}
By the assumption and by Lemma \ref{lem short long}, there exists a continuous supersolution
$u_1$ to \eqref{PMA_sol} in $[0, T)\times\overline{\Omega}$ such that 
$u_0(z)\leq u_1(0, z)\leq u_0(z)+\epsilon$ for all $z\in\Omega$.

Let $\varphi^{\epsilon}\in C^{\infty}(\R\times\C^n)$ such that
\begin{center}
	$\varphi\leq\varphi^{\epsilon}\leq\varphi+\epsilon,$
\end{center}
in $[0, T]\times\partial\Omega$.

 Let $u_2\in C([0, T]\times\bar{\Omega})$ such that
$u_2=\varphi^{\epsilon}$ 
in $[0, T]\times\partial\Omega$ and $u_2(t, .)$
is maximal plurisubharmonic in $\Omega$ for every $t\in [0, T]$. 

Then 
$\tilde{u}=\min\{u_1, u_2\}$ is
a continuous $\epsilon$-superbarrier for 
\eqref{PMA_sol} in $[0, T)\times\overline{\Omega}$.
\end{proof}

For $j=1,2$, we assume that
\begin{itemize}
	\item $F_j(t, z, r)$ is continuous in $[0, T]\times\bar{\Omega}\times\R$ and non-decreasing in $r$.
	\item $\varphi_j(t, z)$ is a continuous function in $[0, T]\times\partial\Omega$
	such that $u_0(z)=\varphi_j(0, z)$ in $\partial\Omega$.
\end{itemize}
We consider the following problems
\begin{equation}\label{PMA1}
\begin{cases}
e^{\dt u+F_1(t, z, u)}\mu(t, z)=\MAu\qquad\mbox{ in }\qquad (0,T)\times\Omega,\\
u=\varphi_1\qquad\mbox{in}\qquad [0, T)\times\partial\Omega,\\
u(0, z)=u_0(z)\qquad\mbox{in}\quad\bar{\Omega},
\end{cases}
\end{equation}
and
\begin{equation}\label{PMA2}
\begin{cases}
e^{\dt u+F_2(t, z, u)}\mu(t, z)=\MAu\qquad\mbox{ in }\qquad (0,T)\times\Omega,\\
u=\varphi_2\qquad\mbox{in}\qquad [0, T)\times\partial\Omega,\\
u(0, z)=u_0(z)\qquad\mbox{in}\quad\bar{\Omega}.
\end{cases}
\end{equation}
\begin{Prop}
	Let $\epsilon_0>0$. If  there exists  a continuous $\epsilon_0$-superbarrier 
	$\overline{u}_1$ for the problem \eqref{PMA1} then for every $\epsilon>\epsilon_0$,
	 there exists a continuous $\epsilon$-superbarrier $\overline{u}_2$ for the problem \eqref{PMA2}. Moreover,
	 if $\overline{u}_1$ is Lipschitz in $t$ then  $\overline{u}_2$ is Lipschitz in $t$.
\end{Prop}
\begin{proof}
	Since  $u_0(z)=\varphi_1(0, z)=\varphi_2(0, z)$ for every  $z\in\partial\Omega$, there exists $\delta>0$ such that
	\begin{center}
		$|\varphi_1 (t, z)-\varphi_2 (t, z)|<\dfrac{\epsilon-\epsilon_0}{3}$,
	\end{center}
	for every $(t, z)\in [0, \delta]\times\partial\Omega$. Let $C>0$ such that
	\begin{center}
		$e^{\dt (\overline{u}_1+Ct)+F_2(t, z, \overline{u}_1+Ct)}\mu(t, z)\geq (dd^c(\overline{u}_1+Ct))^n$,
	\end{center}
	in the viscosity sense in $[0, \delta)\times\partial\Omega$. 
	
	Let $\delta_0=\min\{\delta, \dfrac{\epsilon-\epsilon_0}{3C}\}$. Then $\overline{u}_1+Ct+\dfrac{\epsilon-\epsilon_0}{3}$
	is a continuous $\epsilon$-superbarrier for the problem \eqref{PMA2} 
	in $[0, \delta_0)\times\partial\Omega$. Hence, it
	follows from Proposition \ref{prop short long} that 
	there exists a continuous $\epsilon$-superbarrier $\overline{u}_2$ for the problem \eqref{PMA2}.
\end{proof}
\subsection{The existence of solution}
Now we prove some results about the existence of solution to \eqref{PMA}. Theorem \ref{mainexistence}
 is an immediate corollary of the following:
\begin{The}\label{the exist1}
	Suppose that for every $\epsilon>0$, there exists a continuous 
	$\epsilon$-superbarrier for \eqref{PMA}. Assume that for every $K\Subset\Omega$,
	for every $0<R<S<T$ and $\epsilon>0$, there exists $0<\delta\ll 1$
	such that
	\begin{equation}\label{eq1 exist1}
	(1+\epsilon)f(t+s, z)\geq f(t, z),
	\end{equation}
	and
	\begin{equation}\label{eq2 exist1}
	(1+\epsilon)f(t, z)\geq f(t+s, z),
	\end{equation}
	for all $z\in K$, $0<s<\delta$ and $R<t<S$.
	
	Then \eqref{PMA} admits a unique solution $u$. Moreover, $(u(t, z), \mu(t, z))$ is admissible
	for every $0<t<T$. If \eqref{eq1 exist1} holds in the case $S=T$ then $u$ can be extended continuously
	to $[0, T]\times\overline{\Omega}$ and $(u(T, z), \mu(T, z))$ is admissible.
\end{The}
\begin{proof}
	Denote by $u$ the supremum of continuous subsolutions to \eqref{PMA}.
	Then, by Proposition \ref{prop.subbarrier} and Lemma \ref{lem perron}, $u=u_*$ is a
	supersolution to \eqref{PMA} and $u^*$ is a subsolution to \eqref{PMA}.
	
	By Theorem \ref{the weak compa}, for every $(t, z)\in (0, T)\times\Omega$,
	for each $\epsilon>0$, there exists $\delta>0$ such that
	\begin{center}
		$u(t, z)+\epsilon\geq u^*(s, z)\geq u(s, z)$,
	\end{center}
	for all $0<|s-t|<\delta$. Hence, by using the fact
	that $u$ is lower semi-continuous, we get that $u$ is continuous in $t$. 
	
	Then, by Theorem \ref{the weak compa},
	for every $\epsilon>0$ and
	 $(t, z)\in (0, T)\times\Omega$,
	\begin{center}
		$u^*(t, z)\leq \lim\limits_{s\rightarrow t}u(s, z)+\epsilon=u(t, z)+\epsilon.$
	\end{center}
	Letting $\epsilon\rightarrow 0$, we obtain $u^*=u$.
	
	Thus $u$ is a solution to \eqref{PMA}. The uniqueness of solution holds due to Theorem
	\ref{the weak compa}.
	
	Now, we consider the case where \eqref{eq1 exist1} holds in the case $S=T$. For every
	$t>T$, we define
	\begin{center}
	$\mu (t, .)=\mu (T, .)$, $F(t, . , .)=F(T, ., .)$ and $\varphi (t, .)=\varphi (T, .)$.
	\end{center}
	Denote by $\tilde{u}$ the supremum of continuous subsolutions to the problem
	\begin{equation}\label{PMA extend}
	\begin{cases}
	e^{\dt u+F(t, z, u)}\mu(t, z)=\MAu\qquad\mbox{ in }\qquad (0,\infty)\times\Omega,\\
	u=\varphi\qquad\mbox{in}\qquad [0, \infty)\times\partial\Omega,\\
	u(0, z)=u_0(z)\qquad\mbox{in}\quad\bar{\Omega},
	\end{cases}
	\end{equation}
	Then, it follows from Lemma \ref{lem short long} that $\tilde{u}=u$ in
	$[0, T)\times \bar{\Omega}$. By the continuity of $u$, we also have
	$\tilde{u}^*=u$ in 	$[0, T)\times \bar{\Omega}$. 
	
	By Lemma \ref{lem perron}, $\tilde{u}^*$ is a supersolution to \eqref{PMA extend}.
	Applying the part a) of Theorem \ref{the weak compa}, for every $\epsilon>0$, there
	exists $\delta_1>0$ such that
	\begin{equation}\label{eq1 proofexist1}
	u(t, z)=\tilde{u}^*(t, z)<\tilde{u}(s, z)+\epsilon=u(s, z)+\epsilon,
	\end{equation}
	for every $T-\delta_1<s<t<T$ and $z\in\Omega$.
	Choose $\rho\in C^2(\overline{\Omega})\cap PSH(\Omega)$ such that $-1\leq\rho\leq 0$ and
	$dd^c\rho>0$. Choose $C\gg 1$ such that 
	\begin{center}
		$(dd^c\rho)^n\geq e^{-C+\sup F(., ., \sup\varphi)}\sup fdV$.
	\end{center}
Here, we recall that $\mu(t, z)=f(t, z)dV$.

	Then, for every $1>\epsilon>0$, there exists $0<\delta_2<\delta_1$ such that, for every
	$s\in (T-\delta_2, T)$, the function
	\begin{center}
	$u(s, z)+\dfrac{\epsilon}{3}\rho (z)
	-(C+n\log\dfrac{2}{\epsilon})(t-s)-\dfrac{\epsilon}{3}$,
	\end{center}
is a subsolution to
\begin{equation}
e^{\dt w+F(t, z, w)}\mu(t, z)=(dd^c w)^n,
\end{equation}
in $(s, T)\times\Omega$.

Choose $\delta_2\ll 1$ such that $|\varphi (t, z)-\varphi (s, z)|<\epsilon/3$
for every $z\in\partial\Omega$ and $t, s\in [T-\delta_2, T]$. Using Theorem
\ref{compa.the}, for every $T-\delta_2<s<t<T$ and $z\in\Omega$,
\begin{equation}
u(s, z)+\dfrac{\epsilon}{3}\rho (z)-(C+n\log\dfrac{2}{\epsilon})t-\dfrac{\epsilon}{3}
<u( t, z).
\end{equation}
Choose $0<\delta_3<\delta_2$ such that
 $(C+n\log\dfrac{2}{\epsilon})\delta_3<\dfrac{\epsilon}{3}$. We have
 \begin{equation}\label{eq2 proofexist1}
 u(s, z)-\epsilon<u(t, z),
 \end{equation}
 for every $T-\delta_3<s<t<T$ and $z\in\Omega$.
 
 Combining \eqref{eq1 proofexist1} and \eqref{eq2 proofexist1}, we get that
 $u(t, z)$ converges uniformly to a continuous plurisubharmonic function $u(T, z)$
 as $t\nearrow T$. By using the condition \eqref{eq1 exist1} and applying
 Lemma \ref{rglzingint.lem}, we obtain that $(u(T, z), \mu(T, z))$ is admissible.
\end{proof}
\begin{Cor}\label{cor:exits_1}
	If there exist $0\leq f_0, f_1\in C(\Omega)$ such that $(u_0, f_0dV)$ is admissible and
	$f(t, z)=tf_1(z)+(T-t)f_0(z)$ then \eqref{PMA} has a unique solution.
\end{Cor}

\begin{Rem}\label{rem 2admissible}
	 By combining Theorem \ref{the exist1} and Lemma \ref{rglzingint.lem}, 
	if $\mu$ is independent of $t$ then  \eqref{PMA} admits a solution iff $(u_0, \mu)$
	is  admissible. Hence, by Corollary \ref{cor exist admi}, Definition \ref{def admissible}
	is equivalent to the definition of the admissibility in \cite{EGZ15b}.
\end{Rem}
\begin{The}\label{the remove}
	Let $u\in C^0([0, T]\times\overline{\Omega})$ such that 
	\begin{equation}\label{PMA remove}
	e^{\dt u+F(t, z, u)}\mu(t, z)=\MAu,
	\end{equation}
	in $((0, S)\cup (S, T))\times\Omega$ in the viscosity sense for some $S\in (0, T)$.
	
	If $(u (S, z), \mu (S, z))$ is admissible then $u$ satisfies \eqref{PMA remove}
	in the viscosity sense in $(0, T)\times\Omega$.
\end{The}
\begin{proof}
	By Lemma \ref{lem.seq}, it remains to show that,
	 for every $\epsilon>0$, there exist a subsolution
	 $\underline{u}_{\epsilon}$ and a supersolution $\overline{u}_{\epsilon}$ to
	 \eqref{PMA remove} in $(0, T)\times\Omega$ such that 
	 \begin{center}
	 	$|u(t, z)-\underline{u}_{\epsilon}(t, z)|<\epsilon$ 
	 	and $|u(t, z)-\overline{u}_{\epsilon}(t, z)|<\epsilon$,
	 \end{center}
 for every $(t, z)\in (0, T)\times\Omega$. 
 
 Let $\rho\in C^2(\overline{\Omega})\cap PSH(\Omega)$ such that $-1\leq \rho\leq 0$ and
 $dd^c\rho>0$. Let $C_1\gg 1$ such that
 \begin{center}
 	$(dd^c\rho)^n\geq e^{-C_1+\sup F(., ., \sup u)}\sup f dV.$
 \end{center}
Here, we recall that $\mu(t, z)=f(t, z)dV$.

Denote, for every $0\leq t\leq S$,
\begin{center}
	$h(t)=\sup\{|u(S-s, z)-u(S, z)|: z\in\Omega, s\in [0, t] \}$.
\end{center}
Then $h$ is a non-decreasing function with $h(0)=0$.
 Moreover, by Lemma \ref{lem continuity}, $h$ is continuous.

Let $1>\epsilon>0$. Then there exists $\delta_1\in (0, S)$ such that 
\begin{center}
	$\max\{h(\delta_1), (C_1+n\log\dfrac{3}{\epsilon})\delta_1 \}<\dfrac{\epsilon}{3}.$
\end{center}
For $(t, z)\in (0, S)\times\overline{\Omega}$, we denote
\begin{center}
	$\tilde{h}(t,z)= h(S-t)-(C_1+n\log\dfrac{3}{\epsilon})(t-S+\delta_1)
	+\dfrac{\epsilon(t-S+\delta_1)}{3\delta_1}\rho (z).$
\end{center}

For every $(t, z)\in [0, T)\times\overline{\Omega}$, we define

$$h_{\epsilon}(t, z)=\left\{\begin{matrix}
h(\delta_1)  &   \mbox{ if }t\in [0, S-\delta_1],\\ 
\tilde h(t,z) &  \mbox{ if } t\in [S-\delta_1, S], \\ 
 \dfrac{\epsilon \rho (z)}{3}-(C_1+n\log\dfrac{3}{\epsilon})\delta_1 &  \mbox{ if } t\in [S, T),
\end{matrix}\right.$$

and
\begin{center}
	$\underline{u}_{\epsilon}(t, z)=u(t, z)+h_{\epsilon}(t, z)-\dfrac{\epsilon}{3}.$
\end{center}

Then 
$|u(t, z)-\underline{u}_{\epsilon}(t, z)|<\epsilon$ for every 
$(t, z)\in (0, T)\times\Omega$ and we will show that  $\underline{u}_{\epsilon}$ is 
a subsolution to \eqref{PMA remove} in $(0, T)\times\Omega$.

By the assumption and by the fact that $h_{\epsilon}(t, z)$ is non-increasing in $t$ and plurisubharmonic in $z$,
 we have $\underline{u}_{\epsilon}$ is a subsolution to \eqref{PMA remove}
 in $((0, S)\cup (S, T))\times\Omega$.
 It remains to show that, for every $z_0\in\Omega$, for every upper test function
$p$ for $\underline{u}_{\epsilon}$ at $(S, z_0)$,
\begin{equation}\label{eq test S}
(dd^c p(S, z_0))^n\geq e^{\dt p(S, z_0) +F(S, z_0, p(S, z_0))}\mu(S, z_0).
\end{equation}
Note that, by the definition of $h(t)$, we have
	$u(S-t, z)+h(t)\geq u(S, z),$
for every $(t, z)\in (0, \delta_1)\times\Omega$. Then,
\begin{center}
	$\underline{u}_{\epsilon}(S-t, z)
	\geq \underline{u}_{\epsilon}(S, z)+(C_1+n\log\dfrac{3}{\epsilon})t-t\rho (z)
		\geq \underline{u}_{\epsilon}(S, z)+(C_1+n\log\dfrac{3}{\epsilon})t$,
\end{center}
for every $(t, z)\in (0, \delta_1)\times\Omega$. Hence, for every $z_0\in\Omega$, for every upper test function
$p$ for $\underline{u}_{\epsilon}$ at $(S, z_0)$, we have
\begin{equation}\label{eq t test S}
\dt p(S, z_0)\leq -(C_1+n\log\dfrac{3}{\epsilon}).
\end{equation}
By \cite[Corollary 3.7]{EGZ15b}, $u(t, z)$ is plurisubharmonic in $z$ for every $t\in (0, S)$.
Moreover, since $u$ is continuous on the compact set $[0, S]\times\overline{\Omega}$,
 we have $u(t, z)$ converges
uniformly to $u(S, z)$ as $t\nearrow S$. 
 Then $u(S, z)=\lim\limits_{t\to S^-}u(t, z)$
is plurisubharmonic. Since $p(S, z)-h_{\epsilon}(S, z)+\dfrac{\epsilon}{3}$ is a upper test function
for $u(S, z)$ at $z_0$, we have, by \cite[Proposition 1.3]{EGZ},
\begin{center}
$dd^c(p(S, z)-h_{\epsilon}(S, z))\geq 0$,
\end{center}
at $z_0$. Hence, we have, at $(S, z_0)$
\begin{equation}\label{eq z test S}
	(dd^cp)^n\geq (dd^ch_{\epsilon})^n=\dfrac{\epsilon^n}{3^n}(dd^c\rho)^n
	\geq \dfrac{\epsilon^n}{3^n}(dd^c\rho)^ne^{-C_1+\sup F(., ., \sup u)}\sup f dV.
\end{equation}
Combining \eqref{eq t test S} and \eqref{eq z test S}, we get \eqref{eq test S}. Thus,
 $\underline{u}_{\epsilon}$ is 
a subsolution to \eqref{PMA remove} in $(0, T)\times\Omega$.

Now, we construct $\overline{u}_{\epsilon}$.
Since $(u(S, z), \mu (S, z))$ is admissible, there exist $C_2>0$ and 
$u^{\epsilon}\in C(\overline{\Omega})\cap PSH(\Omega)$ such that
$u(S, z)+\dfrac{\epsilon}{3}< u^{\epsilon}(z)< u(S, z)+\dfrac{2\epsilon}{3}$ and
$(dd^cu^{\epsilon})^n\leq e^{C_2+\inf F(., ., \inf u)}\mu (S, z)$.

Let $\delta_2>0$ such that $C_2\delta_2<\dfrac{\epsilon}{3}$ and 
$\dfrac{\epsilon}{3}<u(t, z)-u^{\epsilon}(z)<\dfrac{2\epsilon}{3}$ for every
$(t, z)\in (S-\delta_2, S)\times\Omega$.
For every $(t, z)\in [0, T)\times\overline{\Omega}$, we define  the function $\overline{u}_{\epsilon}$ by
\begin{center}
	$
	\begin{cases}
	\begin{array}{ll}
	u(t, z) &\mbox{ if } t\in [0, S-\delta_2],\\
	\min\{u(t, z)+\dfrac{\epsilon}{\delta_2}(t-S+\delta_2),
	 u^{\epsilon}(t, z)+C_2(t-S+\delta_2)+\dfrac{2\epsilon}{3}\}
	&\mbox{ if } t\in [S-\delta_2, S],\\
	\min\{u^{\epsilon}(t, z)+C_2(t-S+\delta_2)+\dfrac{2\epsilon}{3}, u(t, z)+\epsilon \}
	&\mbox{ if  } t\in [S, T).
	\end{array}
	\end{cases}$
\end{center}
Then $\overline{u}_{\epsilon}$ is 
a supersolution to \eqref{PMA remove} in $(0, T)\times\Omega$ satisfying 
$|u(t, z)-\underline{u}_{\epsilon}(t, z)|<\epsilon$ for every 
$(t, z)\in (0, T)\times\Omega$. 
The proof is completed.
\end{proof}
By using Theorem \ref{the exist1} and Theorem \ref{the remove}, we have the following:
\begin{The}\label{the exist2}
 Assume that there exist $t_1, ..., t_m$ with $0=t_0<t_1<...<t_m=T$ satisfying
	\begin{itemize}
		\item [i) ]For every $K\Subset\Omega$, for every $\epsilon>0$
		and $t_{k-1}<R<t_k (1\leq k\leq m)$, there exists $0<\delta\ll 1$
		such that
	\begin{equation}\label{eq1 exist2}
	(1+\epsilon)f(t+s, z)\geq f(t, z), \
	\end{equation}
	for all $z\in K, 0<s<\delta,
	R<t<t+s<t_k$.
	\item[ii)] For every $K\Subset\Omega$, for every $\epsilon>0$
	and $t_{k-1}<R<S<t_k (1\leq k\leq m)$, there exists $0<\delta\ll 1$	such that
	\begin{equation}\label{eq2 exist2}
	(1+\epsilon)f(t, z)\geq f(t+s, z), 
	\end{equation}
	for all $z\in K, 0<s<\delta, R<t<t+s<S$.
		\end{itemize}
		If  for every $\epsilon>0$, there exists a continuous $\epsilon$-superbarrier for
	\eqref{PMA_sol} then there exists a
	unique solution to \eqref{PMA_sol}.
	
	In particular, if $(u_0(z), \mu (0, z))$ is admissible then \eqref{PMA_sol} admits a 
	unique solution.
\end{The}
\begin{proof}
	By applying Theorem \ref{the exist1} and using induction, there exists a unique
	function $u\in C([0, T)\times\overline{\Omega})$ satisfying
	\begin{itemize}
		\item $e^{\dt u+F(t, z, u)}\mu(t, z)=\MAu$ in $(t_{k-1}, t_k)\times\Omega$
		in the viscosity sense for $k=1, ..., m$.
		\item  $u=\varphi$ in $(0, T)\times\partial\Omega$.
		\item $u (0, z)=u_0(z)$ for every $z\in\Omega$.
	\end{itemize}
Moreover, $(u(t_k, z), \mu (t_k, z))$ is admissible for every $k=1,..., m-1$. 
Then, by using Theorem \ref{the remove}, we get that $u$ is the unique solution
to \eqref{PMA_sol}.
\end{proof}
\begin{Cor}\label{cor:exits_2}
	Assume that
	there exist  $0=t_0<t_1<...<t_m=T$ and $f_0, ..., f_m\in C(\Omega)$ such that
	\begin{itemize}
		\item $0\leq f_0\leq f_1\leq ...\leq f_m$;
		\item $(u_0, f_0dV)$ is admissible;
		\item $f(t, z)=\dfrac{(t-t_{k-1})f_{k}(z)+(t_k-t)f_{k-1}(z)}{t_k-t_{k-1}}$
		for every $k=1,...,m$ and $(t, z)\in [t_{k-1}, t_k]\times\Omega$.
	\end{itemize}
	Then \eqref{PMA} has a unique solution.
\end{Cor}
\subsection{The proof of Theorem \ref{mainadmissible}}
The first conclusion of Theorem \ref{mainadmissible}
 holds due to Remark \ref{rem 2admissible}.
We now prove that admissibility is local.	
\begin{Prop}\label{admissible2}
	Let $g\geq 0$ be a bounded continuous function in $\Omega$ and $\nu=gdV$.
	Let $\phi\in PSH(\Omega)\cap C(\bar{\Omega})$. 
	If for every $z\in\Omega$, there exists an open neighborhood $U$ of $z$
	such that $(\phi, \nu)$ is admissible in $U$ then $(\phi, \nu)$
	is admissible in $\Omega$. 
\end{Prop}
\begin{proof}
Let $\rho\in C^2(\bar{\Omega})\cap PSH(\Omega)$ such that $\rho|_{\partial\Omega}=0$,
$\nabla\rho|_{\partial\Omega}\neq 0$
and $dd^c\rho\geq dd^c |z|^2$. For every $r>0$, we define
\begin{center}
	$\Omega_{r}=\{z\in\Omega| \rho (z)<-r\}$.
\end{center}
We will show that for all $r>0$, 
if $\Omega_{2r}\neq\emptyset$ then $(\phi, \nu)$ is  admissible in $\Omega_{2r}$.

By the assumption and by the compactness of $\overline{\Omega_{2r}}$, there exist
balls $B(p_1, r_1),...,B(p_m, r_m)$ such that
\begin{itemize}
	\item $B(p_k, 2r_k)\subset\Omega_{r}$ for all $k=1, ..., m$;
	\item  $(\phi, \nu)$ is  admissible in $B(p_k, 2r_k)$ for all $k=1, ..., m$;
	\item $\overline{\Omega_{2r}}\subset\bigcup\limits_{k=1}^mB(p_k, r_k)$.
\end{itemize}
Let $\epsilon>0$. For every $k=1,...,m$, there exists $C_{\epsilon, k}>0$
and $\phi_{\epsilon, k}\in C(B(p_k, 2r_k))$ such that $\phi\leq \phi_{\epsilon, k}\leq
 \phi+\epsilon$
and $(dd^c\phi_{\epsilon, k})^n\leq e^{C_{\epsilon, k}}\nu$ in the viscosity sense in
$B(p_k, 2r_k)$.

We define
\begin{center}
	$\phi_{\epsilon}(z)=
	\min\{\phi_{\epsilon, k}(z)+\dfrac{\epsilon|z-p_k|^2}{r_k^2}: k\in\{1,...,n\}$ 
	such that $|p_k-z|<2r_k\}
	-\dfrac{\epsilon|z|^2}{\min\limits_{0\leq k\leq m}r_k^2}.$
\end{center}
Then $\phi_{\epsilon}$ is a continuous function in $\Omega_{2r}$ satisfying
 $\phi-\dfrac{\epsilon\max_{\overline{\Omega}}|z|^2}{\min_{0\leq k\leq m}r_k}
 \leq \phi_{\epsilon}
 \leq \phi+\epsilon$
 and $(dd^c\phi_{\epsilon})^n\leq e^{C_{\epsilon}}\nu$ in the viscosity sense in 
 $\Omega_{2r}$, where $C_{\epsilon}=\max_{1\leq k\leq m}C_{\epsilon, k}$.  Hence, $(\phi, \nu)$ is  admissible in $\Omega_{2r}$.
 
 \medskip
 Now, let $\tilde{\phi}\in C(\bar{\Omega})\cap PSH(\Omega)$ such that $\phi=\tilde{\phi}$ in
 $\partial\Omega$ and $(dd^c\tilde{\phi})^n=0$ in $\Omega$.   Since $\phi, \tilde{\phi}$ are continuous, for every $\epsilon>0$,
 there exists $r_1>0$ such that $\phi\leq\tilde{\phi}\leq \phi+\dfrac{\epsilon}{5}$ in
 $\Omega\setminus\Omega_{r_1}$.
 
 Let $0<r_2<\dfrac{r_1}{5}$. Since $(\phi, \nu)$ is  admissible in $\Omega_{r_2}$,
  there exist $\phi_{\epsilon}\in C(\Omega_{r_2})$ and $C_{\epsilon}>0$ such that
 $\phi+\dfrac{3\epsilon}{5}\leq \phi_{\epsilon}\leq \phi+\dfrac{4\epsilon}{5}$
 and $(dd^c u_{\epsilon})^n\leq e^{C_{\epsilon}}\mu$ in the viscosity sense 
 in $\Omega_{r_2}$.
 We define
 \begin{center}
 	$\phi_{0, \epsilon}=
 	\begin{cases}
 	\tilde{\phi}-\dfrac{\epsilon\rho}{r_1}\qquad\mbox{in}\quad\Omega\setminus\Omega_{r_2},\\
 	\min\{\tilde{\phi}-\dfrac{\epsilon\rho}{r_1}, \phi_{\epsilon}\}
 	\qquad\mbox{in}\quad \Omega_{r_2}.
 	\end{cases}$
 \end{center}
Then $\phi_{0, \epsilon}$ is a continuous function in $\Omega$ satisfying
$\phi\leq \phi_{0, \epsilon}\leq \phi+\epsilon$ and
 $(dd^c\phi_{0, \epsilon})^n\leq e^{C_{\epsilon}}\nu$ in the viscosity sense in $\Omega$.
 Hence,  $(\phi, \nu)$ is  admissible in $\Omega$.
 \end{proof}
\begin{Prop}
	Let $g\geq 0$ be a bounded continuous function in $\Omega$ and $\nu=gdV$.
	Let $\phi\in PSH(\Omega)\cap C(\bar{\Omega})$.
	 If $\int\limits_{\{g=0\}}(dd^c\phi)_P^n=0$ then $(\phi, \nu)$ is admissible.
\end{Prop}
\begin{proof}
	The problem is local by Proposition \ref{admissible2}. 
	
	Let $B\Subset\Omega$ be a ball. Let $\{U_j\}_{j=1}^{\infty}$ 
	be a decreasing sequence of open subsets
	of $B$ such that
	\begin{itemize}
		\item $\{g=0\}\cap B\subset U_j$ for all $j\in\Z^+$.
		\item $\int_{\overline{U_j}}(dd^c\phi)_P^n<\dfrac{1}{2^j}$ for all $j\in\Z^+$.
	\end{itemize}
Let $\phi_j\in C^{\infty}(\bar{B})\cap PSH(B)$ such that 
$\phi+\dfrac{1}{2^{j+1}}\leq \phi_j\leq\phi+\dfrac{1}{2^j}$ in $B$ and
 $\int_{U_j}(dd^c\phi_j)^n<\dfrac{1}{2^{j+1}}$. For any $j$, we define by $\psi_j$ 
 the solution of 
 \begin{equation}
 \begin{cases}
 \psi_j\in C(\bar{B})\cap PSH(B),\\
 (dd^c\psi_j)_P^n=\chi_{B\setminus U_j}(dd^c\phi_j)^n\qquad\mbox{in}\quad B,\\
 \psi_j=\phi_j\qquad\mbox{in}\quad\partial B.
 \end{cases}
 \end{equation}
 Then $\psi_j\geq \phi_j\geq\phi$ and by \cite[Lemma 1]{Xin}, for any $\epsilon>0$,
 \begin{center}
 	$\lim\limits_{j\to\infty}Cap(\{\psi_j>\phi_j+\epsilon\}, B)=0,$
 \end{center}
where
\begin{center}
	$Cap (U, \Omega)=\sup\{\int\limits_U(dd^cw)^n: w\in PSH(\Omega), 0\leq w\leq 1 \},$
\end{center}
for every Borel set $U\subset\Omega$.

Hence, for every $\epsilon>0$ and $k>0$,
\begin{center}
	$\phi\leq(\limsup\limits_{j\to\infty}\psi_j)^*\leq\phi_k+\epsilon\leq\phi+\dfrac{1}{2^k}
	+\epsilon$.
\end{center}
By Hartogs' lemma, $\psi_j$ is uniformly convergent to $\phi$ in $\bar{B}$. Moreover,
$(\psi_j, \nu)$ is admissible in $B$ for all $j$. Hence, $(\phi, \nu)$ is admissible
in $B$. Thus $(\phi, \nu)$ is admissible in $\Omega$.
\end{proof}
\begin{Rem}
	The condition \textit{``$(\phi, gdV)$ is admissible''} does not imply that  $$\int\limits_{\{g=0\}}(dd^c\phi)_P^n=0.$$
	Indeed, if $\Omega$ is the unit ball, $g=\max\{|z|^2-1/2, 0\}$ and $\phi=\log\max\{|z|^2, 1/2\}$ then 
	$(\phi, gdV)$ is admissible since $\phi_m=\log\max\{|z|^2, 1/2+1/m\}$ is uniformly convergent to $\phi$
	as $m\rightarrow\infty$
	and $(\phi_m, gdV)$ is admissible for every $m>0$. But, it is clear that $\int\limits_{\{g=0\}}(dd^c\phi)_P^n>0$.
\end{Rem}
\section{H\"older continuity of solution}
In this section we prove a H\"older regularity for the viscosity solutions to certain degenerate parabolic complex Monge-Amp\`ere equations in smooth bounded strongly pseudoconvex domains.
\begin{Prop} \label{holder_t_1}
		Assume that $u(t, z)$ is a viscosity solution to \eqref{PMA}.
	Suppose that there exist $C>0$ and $0<\alpha<1$  such that 
	\begin{center}
		$|\varphi (t, z)-\varphi (s, z)|\leq C|t-s|^{\alpha}$,
	\end{center}
	for all $z\in\partial\Omega, t,s\in [0, T)$. Then
	there exists $\tilde{C}>0$ depending on $C$, $n$, $\Omega$,
	$\sup\limits_{[0, T)\times\Omega}f$, $\alpha$
	 and $\sup\limits_{[0, T)\times\bar{\Omega}}F(t, z, \sup\varphi)$  such that
	\begin{equation}\label{prop1holder.eq}
		u (t, z)-u (s, z)\geq -\tilde{C}|t-s|^{\alpha},
	\end{equation}
	for all $z\in\Omega, 0\leq s\leq t<T$.
\end{Prop}
\begin{proof}
Since $u$ is bounded, we only need to show 	\eqref{prop1holder.eq} in the case $|t-s|<1$.

Let $0\leq s_0<t_0<T$ such that $t_0-s_0=\delta<1$. 
Let $\rho\in PSH(\Omega)\cap C^2(\bar{\Omega})$ such that 
$\rho|_{\partial\Omega}=0$ and $(dd^c\rho)^n\geq\mu$. Denote 
\begin{center}
	$C_1=\max\limits_{\bar{\Omega}}(-\rho),\qquad
	 C_2=\alpha^{-1}|\sup\limits_{[0, T)\times\bar{\Omega}}F(t, z, \sup\varphi)|
	 +n\sup\limits_{(0, 1)}(-r^{1-\alpha}\log r),$
\end{center}
and
\begin{center}
	$u_{\delta}(t, z)=u(s_0, z)+\delta^{\alpha}\rho-\max\{C, C_2\}(t-s_0)^{\alpha}.$
\end{center}
We have, for every $(t, z)\in (s_0, t_0)\times\Omega$,
\begin{flushleft}
	$\begin{array}{ll}
	\dt u_{\delta}(t, z)=-\dfrac{\alpha\max\{C, C_2\}}{(t-s_0)^{1-\alpha}}
	&\leq -\dfrac{C_2{\alpha}}{(t_0-s_0)^{1-\alpha}}\\
	&\leq -\dfrac{|\sup\limits_{[0, T)\times\bar{\Omega}}F(t, z, \sup\varphi)|
		+n\alpha\sup\limits_{(0, 1)}(-r^{1-\alpha}\log r)}{\delta^{1-\alpha}}\\
		&\leq -\dfrac{|\sup\limits_{[0, T)\times\bar{\Omega}}F(t, z, \sup\varphi)|
		+n\alpha(-\delta^{1-\alpha}\log \delta)}{\delta^{1-\alpha}}\\
	&=-\dfrac{|\sup\limits_{[0, T)\times\bar{\Omega}}F(t, z, \sup\varphi)|}{\delta^{1-\alpha}}
	+n\alpha\log \delta.
	\end{array}$
\end{flushleft}
Therefore, for each $(t_1, z_1)\in (s_0, t_0)\times\Omega$, for every upper test function
$q$ for $u_{\delta}$ at $(t_1, z_1)$, we have
\begin{equation}\label{eq dtq proof holder_1}
\dt q(t_1, z_1)= \dt u_{\delta}(t_1, z_1)\leq 
-\dfrac{|\sup\limits_{[0, T)\times\bar{\Omega}}F(t, z, \sup\varphi)|}{\delta^{1-\alpha}}
+n\alpha\log \delta.
\end{equation}
Note that $q(t_1, z)-\delta^{\alpha}\rho(z)+\max\{C, C_2\}(t_1-s_0)^{\alpha}$ is a 
upper test function for $u_{\delta}(t_1, z)$ at $z_1$. Since $u_{\delta}(t_1, z)$ is plurisubharmonic
\cite[Corollary 3.7]{EGZ15b}, we have, by \cite[Proposition 1.3]{EGZ},
\begin{equation}
dd^c(q(t_1, z)-\delta^{\alpha}\rho(z))|_{z=z_1}\geq 0.
\end{equation}
Then
\begin{equation}\label{eq MAq1 proof holder_1}
dd^cq(t_1, z)|_{z=z_1}\geq \delta^{\alpha}dd^c\rho(z)|_{z=z_1}\geq 0,
\end{equation}
and
\begin{equation}\label{eq MAq2 proof holder_1}
(dd^cq(t_1, z))^n|_{z=z_1}\geq \delta^{n\alpha}(dd^c\rho)^n|_{z=z_1}\geq \delta^{n\alpha}\mu(t_1, z_1).
\end{equation}
Combining \eqref{eq dtq proof holder_1} and \eqref{eq MAq2 proof holder_1}, we get
\begin{equation}\label{eq q proof holder_1}
(dd^cq)^n\geq e^{\dt q+F(t, z, u_{\delta}(t, z))}\mu(t, z),
\end{equation}
at every $(t_1, z_1)\in (s_0, t_0)\times\Omega$.

By \eqref{eq MAq1 proof holder_1} and \eqref{eq q proof holder_1}, we have
\begin{center}
	$(dd^c u_{\delta})^n\geq e^{\dt u_{\delta}+F(t,z,u_{\delta})}\mu,$
\end{center}
in viscosity sense in $(s_0, t_0)\times\Omega$.

 It is straightforward that $u_{\delta}\leq u$
in $\partial_P((s_0, t_0)\times\Omega) $. By Theorem \ref{compa.the}, we have
$u_{\delta}\leq u$ in $(s_0, t_0)\times\Omega$. Hence
\begin{center}
	$u(t_0, z)\geq u_{\delta}(t_0, z)\geq u(s_0, z)-(C_1+\max (C, C_2))\delta^{\alpha},$
\end{center}
for all $z\in\Omega$.

 Thus,
\begin{center}
	$u(t, z)\geq u(s, z)-(C_1+\max (C, C_2))|t-s|^{\alpha},$
\end{center}
for every $z\in\Omega$ and $0\leq s<t<T$ with $t-s<1$. The proof is completed.
\end{proof}

\begin{Prop}\label{holder_t_2}
		Assume that $\mu>0$ and $u(t, z)$ is a viscosity solution to \eqref{PMA}.
	Suppose that there exist $C>0$, $0<\alpha<1$ and $0< \beta< 1/2$ such that 
	\begin{center}
		$|\varphi (t, z)-\varphi (s, z)|\leq C|t-s|^{\alpha}$,
	\end{center}
	for all $z\in\partial\Omega, t,s\in [0, T)$, and 
	\begin{center}
		$|u_0(z)-u_0(w)|\leq C|z-w|^{\beta}, $
	\end{center}
	for all $z, w\in\bar{\Omega}$. Then, there exists $\tilde{C}>0$ such that
	\begin{center}
		$u_0(z)-u (t, z)\geq -\tilde{C}t^{\alpha}$,
	\end{center}
	for all $z\in\Omega, 0\leq t<T$.
\end{Prop}
\begin{proof}
	We define on $\C^n$
	\begin{center}
		$\tilde{u}_0(z)=\max\limits_{\xi\in\bar{\Omega}}(u_0(\xi)-C|z-\xi|^{\beta}), z\in \C^n.$
	\end{center}
Then $\tilde{u}_0=u_0$ in $\bar{\Omega}$ and
\begin{center}
	$|\tilde{u}_0(z)-\tilde{u}_0(w)|\leq C|z-w|^{\beta}, $
\end{center}
for every $z, w\in\C^n$.

 Let $\chi\in C^{\infty}(\C^n, [0, 1])$ such that
$\chi (z)=0$ for every $|z|>2$ and $\int\limits_{\C^n}\chi=1$. For every $\delta>0$, we
denote
\begin{center}
	$u_{\delta, 0}(z)=\chi_{\delta}*\tilde{u}_0(z),$
\end{center}
where $\chi_{\delta}(z)=\dfrac{1}{\delta^{2n}}\chi (\dfrac{z}{\delta})$.

Then, there exists $C_1>0$ depending only on $\chi$ and $C$ such that, for every $\delta>0$
and $z\in\C^n$,
\begin{center}
	$|u_{\delta,0}(z)-\tilde{u}_0|\leq C_1\delta^{\beta}, 
	|Du_{\delta,0}|\leq C_1\delta^{\beta-1}, |D^2u_{\delta,0}|\leq C_1\delta^{\beta-2}.$
\end{center}
Since $\mu>0$, there exists $C_2>0$ depending only on $C_1$ and $\mu$ such that
\begin{center}
	$(dd^cu_{\delta,0})_{+}^n\leq C_2\delta^{-2n+n\beta}\mu$.
\end{center}
 For every $0<\delta<\min\{1, T\}$, we define
\begin{center}
	$u_{\delta}(t, z)=u_{\delta^{\alpha/\beta}, 0}(z)+C_1\delta^{\alpha}
	+\max\{C, C_3\}t^{\alpha}$,
\end{center}
where
\begin{center}
	$C_3=\dfrac{1}{\alpha}(\log C_2+
	|\inf\limits_{[0, T)\times\bar{\Omega}}F(t, z, \inf\varphi)|
	+\dfrac{n\alpha(2-\beta)}{\beta}\sup\limits_{(0,1)}
	(r^{1-\alpha}\log\dfrac{1}{r})).$
\end{center}
It is direct to check that 
\begin{center}
	$(dd^c u_{\delta})^n\leq e^{\dt u_{\delta}+F(t,z,u_{\delta})}\mu,$
\end{center}
in viscosity sense in $(0, \delta)\times\Omega$. Moreover, 
$u_{\delta}\geq u$
in $\partial_P(0, \delta)\times\Omega) $. Then, by the comparison principle,
$u_{\delta}\geq u$ in $(0, \delta)\times\Omega$. In particular, for every $z\in\Omega$,
\begin{center}
$u(\delta, z)\leq u_{\delta}(\delta, z)=u_{\delta^{\alpha/\beta}, 0}(z)+C_1\delta^{\alpha}
+\max\{C, C_3\}\delta^{\alpha}\leq u_0(z)+(2C_1+\max\{C, C_3\})\delta^{\alpha}$.
\end{center}
Since $0<\delta<\min\{1, T\}$ is arbitrary, we get
\begin{center}
	$u(t, z)\leq u_0(z)+(2C_1+\max\{C, C_3\})t^{\alpha}$,
\end{center}
for every $(t, z)\in (0, \min\{1, T\})\times\Omega$. Since $u$ is bounded, there exists $C_4>0$ depending only on
$C_1, C_3, C, \sup |u|$ such that
\begin{center}
	$u(t, z)\leq u_0(z)+C_4t^{\alpha}$,
\end{center}
for every $(t, z)\in (0, T)\times\Omega$.
The proof is completed.
\end{proof}
\begin{Prop} \label{holder_z} Assume that $u(t, z)$ is a viscosity solution to \eqref{PMA} with $\mu =dV=(dd^c|z|^2)^n $.
	Suppose that there exist $C_1,C_2>0$,  and $0< \beta< 1/2$ such that 
	\begin{eqnarray*}
		|\varphi (t, z)-\varphi (t, w)|&\leq& C_1|z-w|^{2\beta}, \, \forall z, w\in\partial\Omega, t\in [0, T)\\
		|u_0(z)-u_0(w)|&\leq& C_1|z-w|^{\beta},\, \forall z, w\in\bar{\Omega},
	\end{eqnarray*}
	and  $\forall z\in \partial \Omega, t\mapsto   \varphi (t,z) -C_2t$ is decreasing.  Then there exists $\tilde C>0$ such that 
	$$|u(t,z)-u(t,w)|\leq \tilde C|z-w|^\beta.$$
\end{Prop}	

\begin{proof}
	Let $M=\sup_{[0,T)\times \Omega} F(t,z,u(t,z))$, then $u$ satisfies
	\begin{equation}\label{super}
	(dd^c u)^n \leq e^{\dt u +M}\mu,
	\end{equation}
	in the viscosity sense. 
	Let $v(t,z)$ be  the solution of  the  complex Monge-Amp\`ere  equations
	\begin{equation} \label{eq:subsol}
	\begin{cases}
	(dd^c v)^n= e^{M+C_2}\mu, \\
	v(t,z)=\varphi(t,z)\text{ \, on} \, \partial \Omega ,
	\end{cases}
	\end{equation} 
	where $C_2$ satisfies that $\varphi(t,z)- C_2t$ is decreasing. 
	Then $v(t,z)-C_2 t$ is also the solution of 
	
	\begin{equation} \label{eq:subsol'}
	\begin{cases}
	(dd^c v)^n= e^{M+C_2} \mu, \\
	v(t,z)=\varphi(t,z) -C_2 t \text{ \, on} \, \partial \Omega.
	\end{cases}
	\end{equation}
	Applying  the  global maximum principle  of complex Monge-Amp\`ere operator (see for example \cite[Corollary 3.30]{GZ17})  for $v(t,x)-C_2t $ and   the fact that $\varphi(t,z)- C_2t$ is decreasing, we have
	\begin{equation} \label{bound:phi_t}
	v(t,z)-v(s,z)\leq C_2(t-s), \forall t\geq s.
	\end{equation}
	We now have $v(t,s)-C_2 t$ is decreasing in $t$, so   $v(t,z)$ converges,  as $t\rightarrow 0$, to  a psh function $v_0$ satisfying the equation 
	\begin{equation} \label{eq:subsol2}
	\begin{cases}
	(dd^c v_0)^n=  e^{M+C_2} \mu, \\
	v_0=\varphi(0,z)  \text{ \, on} \, \partial \Omega. 
	\end{cases}
	\end{equation}
	Let $\rho\in C^2({\bar{\Omega}})\cap PSH(\Omega)$ such that $\rho<0 $ on $\Omega$, 
	$\rho|_{\partial\Omega}=0$ and $(dd^c\rho)^n\geq\mu$. We choose $K>0$ such that $v-K(-\rho )^\beta \leq u_0 $ on $\bar \Omega$.  
	It follows from \eqref{bound:phi_t}  and \eqref{eq:subsol'} that  
	\begin{equation}\label{subsol_1}
	(dd^c(v-K(-\rho )^\beta))^n\geq (dd^c v)^n = e^{C_2+M}\mu \geq e^{\partial_t v+M} \mu,
	\end{equation} 
	in the viscosity sense. Combining with \eqref{super} and  
	the parabolic comparison principle  yields  $u\geq v -K(-\rho)^\beta$.   Moreover, we also have that $v(t,\cdot)$ is uniformly  $\beta$-H\"older  in $\overline{\Omega}$ (cf.  \cite{BT1,Cha}) 
	
	\medskip
	For the super-barrier, we use the  fact that the harmonic extension $u_\varphi$ of $\varphi(t,z)$ majorizes $u$ from above. Moreover it follows from classical elliptic regularity that  
	\begin{equation}
	|u_\varphi(t,z )-u_\varphi(t,w)| \leq C |z-w|^\beta, \forall t\in [0,T].
	\end{equation}
	
	Combining  both sub/super barriers implies that  there exists $B>0$ such that
	\begin{equation}\label{holder:boundary}
	\forall z\in \overline{\Omega},  \forall \xi \in \partial \Omega, \quad |u(t,z)-u(t,\xi)|\leq B|z-\xi|^\beta, \forall t\in [0,T). 
	\end{equation}
	Consider $\tau\in \mathbb{C}^n$ small with $|\tau|<1$, the function 
	$$w(t,z)=(1-|\tau|^\beta)u(t,z+\tau) + A_2|\tau|^{\beta} |z|^2-A_1|\tau|^\beta$$
	is defined on $\Omega_\tau=\{z\in \Omega| z+\tau \in \Omega\}$. Here we choose $A_2 = e^{(C_F+M)/n}$, $A_1 =A_2 diam(\Omega) + |u|_{L^\infty} +B$, where $C_F$ 
	is the H\"older constant of $F$:
	$$|F(t,z,r)- F(t,\xi,r)|\leq C_F| z-\xi|^\beta, \forall (t,r)\in [0,T)\times [-\|u\|_{L^\infty},\|u\|_{L^\infty}].$$
	It folllows from (\ref{holder:boundary}) that if $z+\tau \in \partial \Omega$ or $z\in \partial\Omega$ then
	$$ w(t,z)\leq u(t,z) - |\tau|^\beta u(t,z) +B(1-|\tau|^\beta)|\tau|^\beta +A_2diam(\Omega)|\tau|^\beta -A_1|\tau|^\beta \leq u(t,z) .$$
	We now prove that $w(t,z)\leq u(t,z)$ on $\Omega_\tau$. Assume by contradiction that it is not the case, then consider  $U_\tau=\{(t,z)\in [0,T)\times \Omega_\tau | w(t,z) > u(t,z)\}$ .
	
	We will show that $w$ is a subsolution for (\ref{PMA}) on $U_\tau$. For any $(t_0,z_0)$ and $q$ is an upper test for $v$ at $(t_0,z_0)$, then $\tilde{q}:= (1-|\tau|^\beta)^{-1}(q(t,z)-A_2|\tau|^\beta |z|^2 +A_1|\tau|^\beta$ is also a upper test for $u(\cdot,\tau+\cdot)$ at the point $(t_0,z_0)$. 
	
	\medskip
	By the definition of viscosity solution $(dd^c\tilde{ q})^n\geq e^{\dt \tilde{q} +F(t_0,z_0,u(z_0+\tau))}\mu$, so
	\begin{eqnarray*} \label{test_ineq}
		\dt q =(1-|\tau|^\beta)\dt \tilde{q}\leq  (1-|\tau|^\beta)\left( \log \frac{(dd^c \tilde{q} )^n}{\mu}   - F(t_0,z_0+\tau,u(t_0,z_0+\tau)) \right).
	\end{eqnarray*}
	Combining with  the concavity of $\log\det$ yields, at  $(t_0,z_0)$, 
	
	\begin{eqnarray}
	\log \frac{(dd^c q)^n}{\mu} &=&\log \frac{((1-|\tau|^\beta) dd^c \tilde q + |\tau|^\beta  A_2dd^c|z|^2)^n}{\mu} \\ \nonumber
	&\geq & (1-|\tau|^\beta) \log \frac{(dd^c \tilde q)^n}{\mu}  +|\tau|^\beta \log \frac{A_2^n (dd^c|z|^2)^n}{\mu} \\ \nonumber
	&=& \partial_t q +(1-|\tau|^\beta) F(t_0,z_0+\tau, u(t_0,z_0+\tau))	 +|\tau|^\beta\log A_2^n,
	\end{eqnarray}
	here for the first inequality  we use the concavity of log det  and for the last equality we  use the estimate for $\partial_t q$ above and the assumption that $\mu=(dd^c|z|^2)^n$.

	This implies that 
	\begin{equation} \label{log_conca}
	(dd^c q)^n\geq e^{\partial_t q+ (1-|\tau|^\beta) F(t_0,z_0+\tau, u(t_0,z_0+\tau))	 +|\tau|^\beta\log A_2^n}\mu.
	\end{equation}
	By the monotonicity of $F$ with respect to third variable, on $U_\tau$ we have
	\begin{eqnarray*}
		F(t_0,z_0+\tau,u(t_0,z_0+\tau)) &\geq& F(t_0,z_0+\tau, (1-|\tau|^\beta)u(t_0,z_0+\tau)   + A_2|\tau|^{\beta} |z|^2-A_1|\tau|^\beta )\\
		&=&  F(t_0,z_0+\tau,v(t_0,z_0))\\
		&\geq& F(t_0,z_0+\tau,u(t_0,z_0)).
	\end{eqnarray*}
	Combining this with  the H\"older continuity in the second variable of $F$ and the choice of $A_2$, we get 
	$$ (1-|\tau|^\beta) F(t_0,z_0+\tau,u(t_0,z_0+\tau))+|\tau|^\beta \log A_2^n\geq   F(t_0,z_0,u(t_0,z_0)) .$$
	So it follows from  \eqref{log_conca} that
	$$(dd^cq)^n\geq e^{\dt q+F(t_0,z_0)}\mu.$$
	This implies that $v$ is a viscossity subsolution to \eqref{PMA}  on $U_\tau$. Therefore the comparison principle implies  that
	$ v\leq u $ on $U_\tau$, and we get a contradiction.  Hence  $U_\tau$ is empty. Finally we infer that,
	for all $z\in \Omega$,
	\begin{center}
		$u(t,z+\tau) + A_2|\tau|^\beta |z|^2-A_1|\tau|^\beta\leq u(t,z) $.
	\end{center}
	This implies that $u$ is H\"older in the $z$ variable as required. 
\end{proof}		

\begin{proof}[Proof of Theorem \ref{mainLip}]
The H\"older continuity for $u$ on  the $z$-variable is straightforward from Proposition \ref{holder_z}.  In Proposition \ref{holder_t_2},  replacing $u_0$ by $u_s$ and using the H\"older continuity in the $z$-variable, we infer that,
for  $0\leq s\leq t $ ,
\begin{equation}
		u (t, z)-u (s, z)\leq \tilde{C}|t-s|^{\alpha}. 
	\end{equation}
Combining with Proposition \ref{holder_t_1} implies the H\"older continuity of $u$ as required.

In the case where $\varphi$ is Lipschitz in $t$, by using Proposition \ref{rgltionint1.prop}
and Proposition \ref{rgltionint2.prop}, we obtain that $u$ is locally Lipschitz in $t$
uniformly in $z$.
\end{proof}
\section{Convergence}
In this section, we prove that the viscosity solution  of  a parabolic complex Monge-Amp\`ere equation
 recovers the solution of the corresponding elliptic equation, extending the convergence result in \cite{EGZ15b}.  
\begin{The}\label{mainconvergence1}
	Consider the problem \eqref{PMA}. Assume that $T=\infty$, 
	$\varphi (t, z)\rightrightarrows\varphi_{\infty} (z)$
	as $t\rightarrow\infty$ and
	$F( t, z, r)\rightrightarrows F_{\infty}(z, r)$ in $\bar{\Omega}\times\R$
	as $t\rightarrow\infty$, where $\rightrightarrows $ denotes the uniform convergence.
	
	Assume that $\sup_{t\geq 0}f(t, z)\in L^1(\Omega)$ 
	and $f (t, z)$ converges almost everywhere
	to $f_{\infty}(z)\in L^1(\Omega)$  as $t\rightarrow\infty$.
	If \eqref{PMA} admits a solution $u$ then $u(t, z)$ converges
	in capacity to $u_{\infty}(z)$
	as $t\rightarrow\infty$, where $u_{\infty}$ is the  unique solution of the equation
	\begin{equation}\label{MAinfty_1}
	\begin{cases}
	u_{\infty}\in\mathcal{F}(\Omega, \varphi_{\infty}),\\
	(dd^cu_{\infty})_P^n=e^{F_{\infty}(z, u_{\infty})} f_\infty(z)dV (z)\qquad\mbox{in}\qquad\Omega,
	\end{cases}
	\end{equation}
		where $\mathcal{F}( \Omega, \varphi_{\infty})$ is a Cegrell class (see Definition
	\ref{def ceg class}).
	
	Moreover, if $\sup_{t\geq 0}f(t, z)\in L^{p}(\Omega)$ for some $p>1$ then 
	$u(t, z)$ converges uniformly to $u_{\infty}(z)$
	as $t\rightarrow\infty$.
\end{The}

Here \textit{the uniform convergence  in capacity} means that, for every $\epsilon>0$,
there exists an open set $U\subset\Omega$ such that
\begin{center}
	$Cap (U, \Omega):=\sup\{\int\limits_U(dd^cw)^n: w\in PSH(\Omega), 0\leq w\leq 1 \}<\epsilon,$
\end{center}
 and $u(t, z)$ converges uniformly to $u_{\infty}(z)$ in $\Omega\setminus U$ as
$t\rightarrow\infty$. By the  countable subadditivity of capacity, this is equivalent to the following: For every $\epsilon>0$,
there exist an open set $U\subset\Omega$ and $T>0$ such that $Cap (U, \Omega)<\epsilon$ and $|u(t, z)-u_{\infty}(z)|<\epsilon$
for every $(t, z)\in (T, \infty)\times (\Omega\setminus U)$.

\begin{proof}
	Let $1\gg \epsilon>0$.
	For every $T>0$, we consider the problem
	\begin{equation}\label{PMA T}
	\begin{cases}
	e^{\dt w+F_{\infty}(z, w)}(1+\epsilon^{n+1})\mu_T(z)=(dd^cw)^n\qquad
	\mbox{ in }\qquad (0,\infty)\times\Omega,\\
	w(t, z)=\varphi (T, z)-\epsilon\qquad\mbox{in}\qquad [0, T)\times\partial\Omega,\\
	w(0, z)=u (T, z)-\epsilon\qquad\mbox{in}\quad\bar{\Omega},
	\end{cases}
	\end{equation}
	where $\mu_T(z)=\sup\limits_{t\in [T, T+1]}f(t, z) dV$. It follows from Lemma
	\ref{rglzingint.lem} that $(u (T, z), \mu_T(z))$ is admissible for every $T$.
	Hence, \eqref{PMA T} admits a unique solution $u_T(t, z)$.
	
	 Let $T_1>0$ such that
	\begin{equation}
	|F(t, z, r)-F_{\infty}(z, r)|<\log (1+\epsilon^{n+1}),
	\end{equation}
	for every $(t, z, r)\in [T_1, \infty)\times\Omega\times\R$ and
	\begin{equation}
	|\varphi(t, z)-\varphi_{\infty} (z)|<\epsilon,
	\end{equation}
	for every $(t, z)\in [T_1, \infty)\times\partial\Omega$.
	
	 We will find $T_2>T_1$, $0<\delta\ll 1$ and $\phi\in\mathcal{F}(\Omega)$ with
	 $Cap(\{\phi<-\epsilon\}, \Omega)=O(\epsilon)$ such that
	 $u_{T_2}(t, z)+\phi$ is a subsolution to the problem
	  \begin{equation}\label{PMA1 T2}
	  \begin{cases}
	  e^{\dt w+F_{\infty}(z, w)}(1+\epsilon^{n+1})\sup\limits_{s\in [T_2,T_2+ T']}
	  f(s, z)dV=(dd^cw)^n\qquad
	  \mbox{ in }\qquad (0, T')\times\Omega,\\
	  w(t, z)=\varphi (T_2, z)\qquad\mbox{in}\qquad [0, T')\times\partial\Omega,\\
	  w(0, z)=u (T_2, z)\qquad\mbox{in}\quad\bar{\Omega},
	  \end{cases}
	  \end{equation}
	and $u_{T_2}(t+\delta, z)-\phi+2\epsilon$ is
	a supersolution to the problem
	 \begin{equation}\label{PMA2 T2}
	\begin{cases}
	e^{\dt w+F_{\infty}(z, w)}(1-\epsilon^{n+1})\inf\limits_{s\in [T_2,T_2+ T']}
	f(s, z)dV=(dd^cw)^n\qquad
	\mbox{ in }\qquad (0, T')\times\Omega,\\
	w(t, z)=\varphi (T_2, z)\qquad\mbox{in}\qquad [0, T')\times\partial\Omega,\\
	w(0, z)=u (T_2, z)\qquad\mbox{in}\quad\bar{\Omega},
	\end{cases}
	\end{equation}
	for every $T'>\delta$. 
		
	By Proposition \ref{holder_t_1}, there exists $\delta>0$ such that
	\begin{equation}
	u_T(t+s, z)\geq u_T(t, z)-\epsilon,
	\end{equation}
	for every $t, T>0, z\in\Omega$ and $s\in [0, \delta]$. By Corollary \ref{rgltionint.cor},
	there exists $C_1>0$ such that
	\begin{equation}
	|\dt u_T(t, z)|\leq C_1,
	\end{equation}
	for every $T>0$, $z\in\Omega$ and $t\geq T+\delta$.
	
	By  Lebesgue's dominated convergence theorem, $\sup_{s\geq t}f(s, z)$
	and $\inf_{s\geq t} f(s, z)$ are convergent to $f_{\infty}(z)$ in $L^1(\Omega)$
	as $t\rightarrow\infty$. Hence,
	\begin{center}
		$\lim\limits_{t\to\infty}\int\limits_{\Omega}
		|\sup_{s\geq t}f(s, z)-\inf_{s\geq t} f(s, z)|dV=0$.
	\end{center}
	Let $T_2>T_1$ such that
	\begin{equation}
	\int\limits_{\Omega}
	|\sup_{s\geq T_2}f(s, z)-\inf_{s\geq T_2} f(s, z)|dV<\dfrac{e^{-C_1-C_2}\epsilon^{n+1}}{n!},
	\end{equation}
	where $C_2=\sup F( ., ., \sup \varphi_{\infty} +\epsilon)$.
	
	Let $\phi$ be the unique solution to the equation
	\begin{equation}
	\begin{cases}
	\phi\in\mathcal{F}(\Omega),\\
	(dd^c\phi)_P^n=e^{C_1+C_2}|(1+\epsilon^{n+1})\sup_{s\geq T_2}f(s, z)
	-(1-\epsilon^{n+1})\inf_{s\geq T_2} f(s, z)|dV.
	\end{cases}
	\end{equation}

	Then, by applying Lemma \ref{lem modify} for $u_{T_2}(t+\delta, z)$, $\phi(z)$ and the equation
	\begin{center}
		$e^{\dt w+F_{\infty}(z, w)}(1+\epsilon^{n+1})\mu_{T_2}(z)=(dd^cw)^n$,
	\end{center}
	in $(0, T')\times\Omega$ for all $T'>\delta$, 
	we get that $u_{T_2}(t+\delta, z)-\phi(z)+2\epsilon$
	 is a supersolution to \eqref{PMA2 T2} and $u_{T_2}(t, z)+\phi(z)$ is a subsolution to
	 \eqref{PMA1 T2}.
	
	Note that $u(t+T_2, z)$ is a subsolution to \eqref{PMA2 T2} and a supersolution
	to \eqref{PMA1 T2}. Since $u_{T_2}(t, z)$ is locally Lipschitz in $t$ uniformly in $z$,
	 applying Theorem \ref{compa.the}	and letting $T'\rightarrow\infty$, we get
	\begin{equation}\label{eq1 proof convergence}
	u(t+T_2, z)\leq u_{T_2}(t+\delta, z)-\phi(z)+2\epsilon,
	\end{equation}
	and
	\begin{equation}\label{eq2 proof convergence}
	u(t+T_2, z)\geq u_{T_2}(t, z)+\phi(z),
	\end{equation}
	for every $(t, z)\in (0, \infty)\times\Omega$.
	
	It follows from Theorem 6.2 in
	\cite{EGZ15b} that $u_{T_2}(t, z)$ converges uniformly to the solution
	$\tilde{u}$ of the equation
	\begin{equation}\label{MAinfty_2}
	\begin{cases}
	(dd^cw)^n=e^{F_{\infty}(z, w)}(1+\epsilon^n)\mu_{T_2} (z)\qquad\mbox{in}\qquad\Omega,\\
	w=\varphi_{\infty}-\epsilon\qquad\mbox{in}\qquad\partial\Omega.
	\end{cases}
	\end{equation}
	Hence, by \eqref{eq1 proof convergence} and \eqref{eq2 proof convergence}, there exists
	$T_3>0$ such that, for every $t>T_3$,
	\begin{equation}\label{eq3 proof convergence}
	\tilde{u}(z)-\epsilon\leq u(t, z)\leq \tilde{u}(z)-\phi(z)+3\epsilon.
	\end{equation}
	It is easy to check that $\tilde{u}+\phi$ is a subsolution to \eqref{MAinfty_1}
	and  $u_{\infty}+\phi-\epsilon$ is a subsolution to \eqref{MAinfty_2}. Then
	\begin{equation}\label{eq4 proof convergence}
	\tilde{u}+\phi\leq u_{\infty}\leq \tilde{u}-\phi+\epsilon.
	\end{equation}
	Combining \eqref{eq3 proof convergence} and \eqref{eq4 proof convergence}, we get
	\begin{equation}\label{eq5 proof convergence}
	|u(t, z)-u_{\infty}(z)|\leq -2\phi +3\epsilon,
	\end{equation}
	for every $t>T_3, z\in\Omega$.
	
	Moreover, it follows from Proposition 3.4 in \cite{NP09} that
	\begin{center}
		$Cap (\{\phi<-\epsilon \}, \Omega)\leq (1+n!2C_3e^{C_1+C_2})\epsilon,$
	\end{center}
	where $C_3=\int_{\Omega}\sup_{t>0}f(t, z)fV$.
	
	Hence, $u(t, z)$ converges uniformly
	in capacity to $u_{\infty}(z)$ as $t\rightarrow\infty$.
	
	If $\sup_{t\geq 0}f(t, z)\in L^{p}(\Omega)$ for some $p>1$ then we can choose $T_2$ such that
	\begin{equation}
	\int\limits_{\Omega}
	|\sup_{s\geq T_2}f(s, z)-\inf_{s\geq T_2} f(s, z)|^pdV<\dfrac{e^{-C_1-C_2}\epsilon^{n+1}}{n!}.
	\end{equation}
	Then, by \eqref{eq5 proof convergence} and by using Theorem 1.1 in \cite{GKZ} for
	$\phi$ and $0$, we obtain the uniform convergence of $u( t, z)$ as $t\rightarrow\infty$.
\end{proof}

\end{document}